%% file: Integral_BC_cohomology_CE_threefolds.tex
\documentclass[a4paper, 11pt]{article}
\usepackage{standalone}
\input{preamble.tex}

\usepackage[margin=1in]{geometry}

\title{Integral Bott--Chern cohomology\\of Generalized Calabi--Eckmann threefolds}
\author{Federico Thiella\footnote{Dipartimento di Matematica e Informatica, Università degli Studi di Firenze}}
\date{}

\begin{document}
\maketitle
\begin{abstract}
    We compute the integral Bott--Chern cohomology of a class of generalized Calabi--Eckmann manifolds of complex dimension \(3\). Although these manifolds are mutually non-biholomorphic, we show that their integral Bott--Chern cohomology rings are nevertheless isomorphic. A cohomological distinction is then obtained via a characteristic map in cohomology, connecting singular, Dolbeault, and integral Bott--Chern cohomologies.
\end{abstract}
\section*{Introduction}
Calabi--Eckmann manifolds are among the first examples of compact complex manifolds that are neither algebraic nor Kähler \cite{Calabi-Eckmann}. This property arises from their topology: since they are all diffeomorphic to the product of odd dimensional spheres \(S^{2n-1} \times S^{2m-1}\), their second Betti number vanishes. This class includes all diagonal Hopf manifolds \cite{Hirzebruch1951}, and thus constitutes a natural generalization thereof. The generalization also encompasses the holomorphic bundle structure that any Hopf manifold enjoys. Indeed, it is well known that these manifolds are products of \(S^1\) with a Hopf bundle over a projective space. Similarly, Calabi--Eckmann manifolds are total spaces of a holomorphic principal torus bundle over the product of two projective spaces; the complex structure of the fiber identifies two canonical \(S^1\)-bundles, each over one of the components of the base, each of which is a Hopf bundle.

Following the ideas of Haefliger \cite{Haefliger1985}, J.~J.~Loeb and M.~Nicolau generalized the Calabi--Eckmann construction to a whole class of complex structures on the product of two odd dimensional spheres that is closed under small deformations \cite{LoebNicolau1996}. Their construction involves an open subset of \(\C^N\) and a holomorphic vector field \(X\) in its Poincaré domain. Following \cite{Haefliger1985}, under some additional conditions on \(X\), its flow is transverse to a real subvariety of \(\C^N\) diffeomorphic to \(S^{2n-1} \times S^{2m-1}\); hence by transversality, the latter inherits a natural complex structure that can be investigated by means of the originating vector field \(X\).

In this paper we consider a restricted subclass of Loeb--Nicolau's manifolds; namely those that are LVMB manifolds. This is a different class of compact complex manifolds introduced by S.~López de Medrano, A.~Verjovsky \cite{LopezVerjovsky1997} and L.~Meersseman \cite{Meersseman2000}; some of them---which we will refer to as \emph{regular} LVMB manifolds---admit a holomorphic torus bundle structure over a smooth compact toric variety \cite{MeerssemanVerjovsky2004}. Remarkably, the construction of LVMB manifolds can be reformulated in purely combinatorial terms \cite{Bosio2001}, and much of their geometry can be read directly from the combinatorial datum \cite{MeerssemanVerjovsky2004,PanovUstinovsky2012,BattagliaZaffran2015,Thiella2026preprint}. In this framework, Loeb--Nicolau's manifolds that are also LVMBs are precisely those for which \(X\) is a diagonal vector field. Nonetheless, we will further restrict to the manifolds in the intersection of these classes that admit a holomorphic bundle structure over a smooth Hirzebruch surface \(\hirz{q}\): we refer to these objects as \emph{generalized Calabi--Eckmann threefolds}. This choice allows us to exploit the best features of both constructions. Indeed, by Loeb and Nicolau, we know that these are mutually non-biholomorphic complex manifolds, all diffeomorphic to \(S^3 \times S^3\), thus representing a natural generalization of Calabi--Eckmann manifolds. On the other hand, being regular LVMB manifolds, we can retrieve all the information about their bundle structure directly from the associated combinatorial datum \cite{Thiella2026preprint}.

The argument employed by Loeb and Nicolau to show that generalized Calabi--Eckmann threefolds are mutually holomorphically distinct exploits the strong rigidity of their geometry; namely that any biholomorphism between them is necessarily equivariant. Therefore, one expects the same rigidity to be reflected in at least one of the cohomological theories that a complex manifold admits. Interestingly, the answer is negative as long as one considers only Dolbeault, Aeppli or Bott--Chern cohomologies. Indeed, from the general theory of LVMB manifolds, we have a bigraded model for the algebra of forms \cite[Prop.~2.8]{Thiella2026preprint} (see also \cite{PanovUstinovsky2012}); applying this to generalized Calabi--Eckmann threefolds, we compute their Dolbeault, Aeppli and Bott--Chern cohomology. While the first is constant throughout the whole class, Bott--Chern cohomology perceives a difference between the “standard'' Calabi--Eckmann threefold and all the other ones; the latter are not distinguished by this cohomology. Roughly speaking, the reason is that when the integer parameter \(q\) of the Hirzebruch base space is nonzero, a multiple of the singular cocycle generating \(H^{2,2}_\text{BC}\) becomes trivial in Bott--Chern cohomology. This suggests that a cohomological theory refining Bott--Chern cohomology and in which \(q > 0\) is not a unit might detect a torsion cohomology class that thus would be characteristic of each of these manifolds.

This led us to consider the so-called \emph{integral Bott--Chern cohomology}: a cohomological theory introduced by M.~Schweitzer in \cite{Schweitzer2007preprint}, where the author provides a sheaf-theoretic definition of the usual Bott--Chern cohomology. When \(\underline{\C}\) is replaced by \(\underline{\Z}\) as the sheaf of coefficients, one obtains a natural Bott--Chern cohomology with integer coefficients. Combining the soft resolution of the \emph{Bott--Chern complex} provided by \cite{HarveyLawson2008} with a natural exact sequence connecting integral Bott--Chern cohomology with usual Bott--Chern and de Rham cohomologies, we explicitly compute the integral Bott--Chern cohomology groups \(H^{s,t}_\text{BC}(N_q;\Z)\) of any generalized Calabi--Eckmann threefold \(N_q\). In particular, we deduce the following.
\begin{namedtheorem}[\labelcref{integral_BC_are_iso}]
    Let \(N_q,N_{q'}\) be two generalized Calabi--Eckmann threefolds with \(q,q' > 0\). Then, for any \(s,t\)
    \begin{equation*}
        H^{s,t}_\text{BC}(N_q;\Z) \iso H^{s,t}_\text{BC}(N_{q'};\Z).
    \end{equation*}
\end{namedtheorem}
The reason why these groups fail to distinguish the manifolds \(N_q\) is that, in bidegree \((2,2)\), we have
\begin{equation*}
    H^{2,2}_\text{BC}(N_q;\Z) \iso \frac{\C}{\Z} \oplus \frac{\C}{q\Z}
\end{equation*}
which is isomorphic as a complex abelian Lie group to \(\C^* \times \C^*\) whenever \(q > 0\). Since this is the product of two divisible groups, each of its components contains classes of any positive order; thus the integer \(q\) is not a characteristic constant of \(H^{2,2}_\text{BC}(N_q;\Z)\).

Like usual Bott--Chern cohomology, the integral one also admits a natural multiplicative structure \cite{Schweitzer2007preprint}. Therefore, we have a well-defined integral Bott--Chern cohomology ring \(H^{*,*}_\text{BC}(N_q;\Z)\) admitting a natural ring homomorphism
\begin{equation*}
    H^{*,*}_\text{BC}(N_q;\Z) \longrightarrow H^{*,*}_\text{BC}(N_q;\C),
\end{equation*}
that is not injective in general. Typically, the cohomology algebra is more rigid than the underlying (bi)graded module because its automorphisms also need to preserve the multiplicative structure. However, again due to the divisibility of the middle integral Bott--Chern cohomology groups, this is not the case for generalized Calabi--Eckmann threefolds.
\begin{namedtheorem}[\labelcref{BC_ring_iso_and_algebraa}]
        If both \(q,q' > 0\), there exists a ring isomorphism \(f : H^{*,*}_\text{BC}(N_{q'};\Z) \to H^{*,*}_\text{BC}(N_{q};\Z)\) and a \(\C\)-algebra isomorphism \(f_{\C} : H^{*,*}_\text{BC}(N_{q'};\C) \to H^{*,*}_\text{BC}(N_{q};\C)\) such that the natural square
    \begin{equation*}
        \begin{tikzcd}
            H^{*,*}_\text{BC}(N_{q'};\Z) \ar[r,"f","\iso"'] \ar[d] & H^{*,*}_\text{BC}(N_{q};\Z) \ar[d]\\
            H^{*,*}_\text{BC}(N_{q'};\C) \ar[r,"f_\C","\iso"'] & H^{*,*}_\text{BC}(N_{q};\C)
        \end{tikzcd}
    \end{equation*}
    commutes.
\end{namedtheorem}

These results show that any cohomological invariant capable of distinguishing these spaces must arise as a map connecting singular cohomology to some other cohomological theories. Such a map naturally comes from the very definition of the integral Bott--Chern complex, in which the sheaf of \(\Z\)-valued locally constant functions is mapped to the sum of Dolbeault and anti-Dolbeault complexes. Thus, for \(q >0\), truncating the corresponding long exact hypercohomology sequence, we obtain the following characteristic short exact sequence
\begin{equation*}
    \begin{tikzcd}
        0 \ar[r] & H^3(N_q;\Z) \ar[r, "c"] & H^{1,2}_{\delbar}(N_q) \oplus H^{2,1}_{\del}(N_q) \ar[r] & H^{2,2}_\text{BC}(N_q;\Z) \ar[r] & 0,
    \end{tikzcd}
\end{equation*}
in which the central term is the sum of Dolbeault and anti-Dolbeault cohomology groups. This sequence provides a cohomological invariant that effectively distinguishes the manifolds \(N_q\). Indeed:
\begin{namedtheorem}[\labelcref{GCE_are_distinct_cohomological}]
    Let \(N_q\) and \(N_{q'}\) be two generalized Calabi--Eckmann threefolds. Then \(c \iso c'\) if and only if \(q = q'\).
\end{namedtheorem}

\paragraph{Acknowledgements}
I am deeply grateful to my advisors Daniele Angella and Fiammetta Battaglia for their constant support during the preparation of this paper. Special thanks go to Linda Carnevale for all the ideas she exchanged with me during the years.

The author is supported by the Università degli Studi di Firenze, INdAM--GNSAGA and
PRIN 2022 project “Real and Complex Manifolds: Geometry and holomorphic dynamics” (code
2022AP8HZ9).

\section{Generalized Calabi--Eckmann threefolds}
Calabi--Eckmann manifolds are classical examples of complex structures on the product of two odd dimensional spheres \(S^{2n-1} \times S^{2m-1}\); when \(1 = n < m\), they generalize the standard \(m\)-dimensional Hopf manifold \cite{Calabi-Eckmann}. These manifolds are constructed as follows: for any \(\tau \in \C\) with \(\Im \tau > 0\) we consider the holomorphic action
\begin{align*}
    \alpha : \C \times (\C^n \setminus \{0\}) \times (\C^m \setminus \{0\}) &\longrightarrow (\C^n \setminus \{0\}) \times (\C^m \setminus \{0\})\\
    (t, (\underline{z}, \underline{w})) &\longmapsto ((e^{2\pi i t} z_1, \dots, e^{2\pi i t} z_n), (e^{2\pi i\tau t} w_1, \dots, e^{2\pi i \tau t} w_m)).
\end{align*}
One can readily verify that its orbit space is the Calabi--Eckmann manifold \(M_{n,m;\tau}\) \cite{Calabi-Eckmann,LopezVerjovsky1997}. Furthermore, this admits a second action of the elliptic curve of period \(\tau\) that we denote as \(\cpxTorus{1}_\tau = \C/\left\langle1,\tau\right\rangle\Z\); on classes, this is given as
\begin{align*}
    \beta : \cpxTorus{1}_\tau \times M_{n,m;\tau} &\longrightarrow M_{n,m;\tau}\\
    (v, [(\underline{z},\underline{w})]_{M_{n,m;\tau}}) &\longmapsto [(e^{2\pi i \frac{v}{\tau}} z_1, \dots, e^{2\pi i \frac{v}{\tau}} z_n), \underline{w}]_{M_{n,m;\tau}} = [\underline{z}, (e^{-2\pi i v} w_1, \dots, e^{-2\pi i v} w_m)]_{M_{n,m;\tau}}.
\end{align*}
Since this action is holomorphic and principal, \(M_{n,m;\tau}\) automatically inherits a holomorphic principal \(\cpxTorus{1}_\tau\)-bundle structure over \(\CP^n \times \CP^m\). Indeed, we can check that the bundle map \(\pi : M_{n,m;\tau} \to \CP^n \times \CP^m\) splits, with respect to the lattice of the fiber, into the product of two \(S^1\) bundles, each of which is a Hopf bundle \cite{Calabi-Eckmann}.

A broad generalization of this construction was provided by J.~J.~Loeb and M.~Nicolau in \cite{LoebNicolau1996}. There, they consider the foliation induced by a holomorphic vector field \(X\) on an open subset of \(\C^N\). Under suitable conditions (see \cite[Thm.~1]{LoebNicolau1996}), the leaf space \(S^{n,m}_X\) is a compact complex non-Kähler and non-algebraic manifold diffeomorphic to \(S^{2n-1} \times S^{2m-1}\). This feature sets them as a broad generalization of Calabi--Eckmann manifolds. 
 
While in general the geometry of the manifolds \(S^{n,m}_X\) can be very involved, when \(X\) is diagonal, i.e. in the case \(X\) is of the form
\begin{equation*}
    X = \sum_{j=1}^N \lambda_j z_j \frac{\del}{\del z_j},
\end{equation*}
their description simplifies considerably. Indeed, in this case, these manifolds are examples of LVM manifolds \cite{LopezVerjovsky1997,Meersseman2000}. These manifolds have recently been widely studied by many authors \cite{Bosio2001,MeerssemanVerjovsky2004,Battisti2013,BattagliaZaffran2015,BattistiOeljeklaus2015}, and indeed their investigation has highlighted a richer geometry for these spaces. In particular, the construction of LVM manifolds can be reformulated by means of purely combinatorial terms \cite{Bosio2001}; this actually produces a larger class---the so-called LVMB manifolds---but it also shifts the core of the construction from holomorphic dynamics to combinatorics. Moreover, when the initial datum fulfils a regularity condition, the corresponding LVMB manifold---for this reason called \emph{regular}---admits a holomorphic principal torus bundle structure over a compact smooth toric variety \cite{MeerssemanVerjovsky2004,Battisti2013,BattagliaZaffran2015}. This is the feature of primary interest throughout this paper. Indeed, it has recently been shown that the geometry of the bundle can be read directly from the original combinatorial datum \cite{Thiella2026preprint}; this in particular completely determines the bigraded bidifferential algebra of smooth forms that encodes the full cohomological information of these manifolds.

\subsection{Generalized Calabi--Eckmann threefolds as LVM manifolds}
In general, LVM manifolds of the type described by Loeb and Nicolau may be of any complex dimension; yet, for computational reasons, in this paper we consider only the three-dimensional case, which is the lowest dimension in which these examples are not necessarily Hopf manifolds.

To construct our class, we notice that the product \(\CP^1 \times \CP^1\) is indeed the Hirzebruch surface \(\hirz{0}\) \cite[V\S 4]{BarthHulekPetersVanDeVen2004}; thus the standard Calabi--Eckmann threefold \(M_{1,1;i}\) fitting in the holomorphic torus bundle \(\cpxTorus{1}_i \hookrightarrow M_{1,1;i} \to \CP^1 \times \CP^1\) can be naturally generalized by a regular LVMB manifold \(N_q\) that fits in the bundle \(\cpxTorus{1}_i \hookrightarrow N_q \to \hirz{q}\). We will see below that this manifold is automatically of the type described by Loeb and Nicolau. To construct \(N_q\), we use the approach by Battaglia and Zaffran \cite{BattagliaZaffran2015}, since, from their viewpoint, it is more straightforward to prescribe the base space of such a holomorphic bundle. This is done encoding the datum of a complete smooth fan \(\Sigma\) into a so-called triangulated vector configuration \((V_\Sigma,\calT_\Sigma)\) \cite{BattagliaZaffran2015,BattagliaZaffran2017} and choosing a Gale Transform for it.

In this paper we do not describe this construction in general; for this we refer to \cite{BattagliaZaffran2015,BattagliaZaffran2017} and references therein. Instead, we will directly construct the example we need, following their construction nonetheless. In particular, we will employ the notation of our previous work \cite{Thiella2026preprint}.

We start from the toric construction of Hirzebruch surfaces: consider the rational polyhedral fan \(\Sigma_q\) in \(\R^2\) generated by the vectors
\begin{equation*}
    v_1 = e_1, \quad v_2 = -e_1 + q e_2, \quad v_3 = -e_2, \quad v_4 = e_2
\end{equation*}
as in \cref{HirzebruchFan}. 
\begin{figure}
    \centering
    \begin{tikzpicture}[scale=1]
        \def\r{2}
        
        \begin{scope}
            \fill[blue!50] (0,0) -- (3,0) -- (3,3.5) -- (0,3.5) -- cycle;
            \fill[red!50] (0,0) -- (0,3.5) -- (-1.75,3.5) -- cycle;
            \fill[green!50] (0,0) -- (-1.75,3.5) -- (-3,3.5) -- (-3,-2) -- (0,-2) -- cycle;
            \fill[yellow!50] (0,0) -- (0,-2) -- (3,-2) -- (3,0) -- cycle;
        \end{scope}
        
        \foreach \x in {-2,-1,0,1,2}
        \foreach \y in {-1,0,1,2,3}
        \fill[gray!50] (\x,\y) circle (1pt);
        
        \draw[->, very thick, black] (0,0) -- (1,0) node[anchor=north west] {\footnotesize $v_1 = (1,0)$};
        \draw[->, very thick, black] (0,0) -- (0,1) node[anchor=north west] {\footnotesize $v_4 = (0,1)$};
        \draw[->, very thick, black] (0,0) -- (-1,\r) node[anchor=south east] {\footnotesize $v_2 = (-1,q)$};
        \draw[->, very thick, black] (0,0) -- (0,-1) node[anchor=north east] {\footnotesize $v_3 = (0,-1)$};
        
        \draw[thick, black!60, dashed] (1,0) -- (2.5,0);
        \draw[thick, black!60, dashed] (0,1) -- (0,2.5);
        \draw[thick, black!60, dashed] (-1,\r) -- (-1.4, 2.8);
        \draw[thick, black!60, dashed] (0,-1) -- (0,-1.5);
        
        \fill[black] (0,0) circle (2pt);
    \end{tikzpicture}
    \caption{The fan \(\Sigma_q\) for \(q = 2\)}
    \label{HirzebruchFan}
\end{figure}
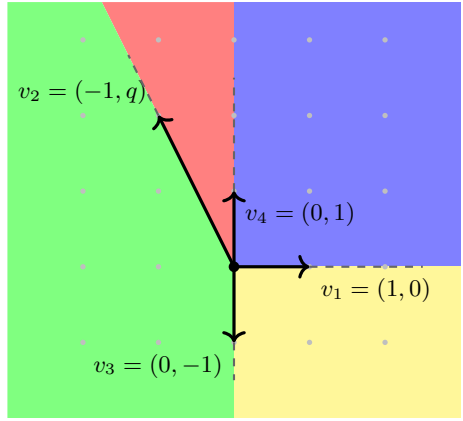
This datum is naturally embedded in the triangulated vector configuration \((V,\calT)\) where \(V = (v_1,v_2,v_3,v_4)\) and
\begin{equation*}
    \calT = \{\{1,4\}, \{2,4\}, \{2,3\}, \{1,3\}\}.
\end{equation*}
For the LVMB construction to apply, such a configuration needs to be odd and balanced. The first condition requires that the space of linear relations among the vectors in \(V\) be of dimension \(2m +1 \geq 3\); the second, instead, that \(\sum_{v \in V} v = 0\) \cite{BattagliaZaffran2015}. In our particular case, \((V,\calT)\) acquires these properties after prepending \(v_0 = -q e_2\) to \(V\); henceforth we continue to denote by \(V\) the new configuration of five vectors. Notice that \(\calT\) has not been modified, hence \(v_0\) does not belong to any \(\tau \in \calT\): vectors like this are called \emph{ghost vectors}.

The next step in the construction requires the choice of a Gale dual configuration for \(V\), this construction is described in detail in \cite{LoeraRambauSantos2010}. While the link between a triangulated vector configuration and its Gale dual is deeper, in practice, it amounts to choosing a basis of the space of linear relations among the vectors in \(V\). In the specific case we are taking into account, this is given by the columns of the matrix
\begin{equation*}
    \hat\Lambda^{\R}_q =
    \begin{bmatrix}
        1 & 0 & 0\\
        1 & 1 & 0 \\
        1 & 1 & 0 \\
        1 & q & 1 \\
        1 & 0 & 1
    \end{bmatrix};
\end{equation*}
notice that, as \(V\) is always assumed to be balanced, the first column can be always taken of this form. Thus, considering the rows of \(\hat\Lambda_q\), we obtain a \emph{graded} configuration of five vectors in \(\R^3\); which, in turn, can be identified with a configuration of five points in the affine space \(\aff_\R^2\) when projected onto the affine plane \(H = \{x_1 = 1\}\). To obtain a configuration of complex points \(\Lambda_q\), and consequently the LVMB datum \((\Lambda_q,\mathcal{E})\) we take the period matrix \(\Pi = [1,i]\) and the rows of the matrix
\begin{equation*}
    \Lambda_q = \Lambda^\R_q \trans{\Pi} = 
    \begin{bmatrix}
        0 & 0\\
        1 & 0 \\
        1 & 0 \\
        q & 1 \\
        0 & 1
    \end{bmatrix}
    \begin{bmatrix}
        1\\
        i
    \end{bmatrix}
    =
    \begin{bmatrix}
        0 \\
        1 \\
        1 \\
        q + i \\
        i
    \end{bmatrix}.
\end{equation*}
While \(\Lambda_q\) is a \(5\)-tuple of points in \(\aff^1_\C\) constructed as above, the virtual chamber \(\mathcal{E}\) is obtained by complementing the subsets in \(\calT\) within \(\{0,\dots, 4\}\). Specifically,
\begin{equation*}
    \mathcal{E} = \{\{0,2,3\}, \{0,1,3\}, \{0,1,4\}, \{0,2,4\}\}.
\end{equation*}
The pair \((\Lambda,\mathcal{E})\) we have constructed determines a unique LVMB manifold as follows: the virtual chamber \(\mathcal{E}\) determines an open subset
\begin{equation*}
    U = \bigcup_{\varepsilon \in \mathcal{E}} \bigcap_{j \in \varepsilon} D(z_j),
\end{equation*}
where \(D(z_j) = \{\underline{z} \in \C^5 : z_j \neq 0\}\) are the distinguished open subsets of the Zariski topology of \(\C^5\). In the particular case we are describing, \(U = (\C^* \times \C^2\setminus \{0\}) \times (\C^2\setminus \{0\})\). On it, we consider the \(\C\)-action
\begin{equation}
    \label{eq:alpha}
    \begin{aligned}
        \alpha : \C \times \proj(U) &\longrightarrow \proj(U)\\
        (t, (z_0 : \dots z_4)) &\longmapsto (z_0 : e^{2\pi i t} z_1 : e^{2\pi i t} z_2 : e^{2\pi i (q+i) t} z_3 : e^{- 2\pi t} z_4)
    \end{aligned}
\end{equation}
whose orbit space \(N_q\) gets a natural structure of compact complex manifold of dimension \(3\) \cite[Prop.~2.1]{BattagliaZaffran2015}. As the original triangulated vector configuration encodes the fan of the Hirzebruch surface \(\hirz{q}\), \(N_q\) is the total space of a holomorphic torus bundle \(\pi : N_q \to \hirz{q}\) \cite{BattagliaZaffran2015,Thiella2026preprint}. In particular, for \(q = 0\) we recover the standard Calabi--Eckmann threefold \(M_{1,1;i} \to \CP^1 \times \CP^1\).

The Loeb--Nicolau construction, instead, can be recognized by observing that \(\alpha\) coincides with the translation along the flow of the vector field
\begin{equation*}
    X_q = \sum_{j=1}^{4} (\Lambda_q)_j z_j  \frac{\del}{\del z_j}
\end{equation*}
acting on \(\proj(U) \iso (\C^2 \setminus \{0\})_{z_1,z_2} \times (\C^2 \setminus \{0\})_{z_3,z_4}\). Note that, to be consistent with the indices appearing in \cref{eq:alpha}, we are considering the list \(\Lambda_q = ((\Lambda_q)_0, \dots, (\Lambda_q)_4)\) indexed from 0 to 4. In particular, \((\Lambda_q)_0\) is dropped when passing to the projectivization of \(U\).
\begin{definition}
    We call \emph{generalized Calabi--Eckmann} threefold of parameter \(q \in \mathbb{N}\) the LVMB manifold \(N_q\) constructed as above. The natural projection \(\pi : N_q \to \hirz{q}\) is a holomorphic principal \(\cpxTorus{1}\)-bundle.
\end{definition}
\begin{remark}
    According to Definition 1 in \cite{Thiella2026preprint}, all the LVMB manifolds \(N_q\) are regular. This is because the initial fan datum \(\Sigma_q\) is smooth for every \(q \in \mathbb{N}\).
\end{remark}

Since the manifolds \(N_q\) are regular LVMB manifolds, we can apply \cite[Thm.~1]{Thiella2026preprint} to compute the characteristic class of the torus bundle \(\pi\). From that result it follows that this invariant can be read off from any left inverse of the matrix \(\hat\Lambda^\R_q\); a possible choice is the matrix 
\begin{equation*}
    \begin{bmatrix}
        1 & 0 & 0 & 0 & 0\\
        -1 & 1 & 0 & 0 & 0\\
        -1 & 0 & 0 & 0 & 1
    \end{bmatrix}.
\end{equation*}
After quotienting out the span of the first row (cf.~\cite[Thm.~1]{Thiella2026preprint}), we obtain the characteristic class of the torus bundle
\begin{equation*}
    \gamma(\pi : N_q \to \hirz{q}) = w_1 \otimes f_1 + w_4 \otimes f_2 \in H^2(\hirz{q};\Z) \otimes H_1(\cpxTorus{1};\Z);
\end{equation*}
where \(f_1,f_2\) are the generators of \(H_1(\cpxTorus{1};\Z)\). On the other hand, as we showed in general in \cite{Thiella2026preprint}, \(w_1,w_4\) are the cohomology classes that are Poincaré-dual to the divisors of \(\hirz{q}\) corresponding to the rays \(\R_{\geq 0} v_1, \R_{\geq 0} v_4\) in the fan \(\Sigma_q\). As a standard fact in toric geometry, these classes generate the whole Dolbeault cohomology ring of \(\hirz{q}\), in particular, it can be presented as
\begin{equation*}
    H^{*,*}(\hirz{q}) \iso \frac{\C[w_1,w_4]}{w_1^2, w_4^2+ q w_1 w_4},
\end{equation*}
where \(w_1, w_4 \in H^{1,1}(\hirz{q}) \cap H^2(\hirz{q};\Z)\). From the topological viewpoint, all the manifolds \(N_q\) are diffeomorphic to \(S^3 \times S^3\). This already emerges from Loeb and Nicolau's construction \cite{LoebNicolau1996} since their manifolds are explicitly obtained as complex structures on the product of two odd dimensional spheres. However, we can also deduce this from the expression of \(\gamma(\pi)\). Indeed, this characteristic class can be naturally identified with an isomorphism in cohomology \(\gamma(\pi) : H^1(\cpxTorus{1};\Z) \to H^2(\hirz{q};\Z)\) sending \(f_1 \mapsto w_1\) and \(f_2 \mapsto w_4\). Therefore, by \cite[\S13.2]{Hofer1993}, \(N_q\) is diffeomorphic to \(S^3 \times S^3\) for any \(q \in \mathbb{N}\). This behavior differs from that of Hirzebruch surfaces: indeed \(\hirz{q} \approx_\infty S^2 \times S^2\) when \(q\) is even, and \(\hirz{q} \approx_\infty \CP^2 \# \overline{\CP^2}\) for \(q\) odd \cite{Hirzebruch1951}.

\subsection{\texorpdfstring{Cohomology of Generalized Calabi--Eckmann threefolds}{Cohomology of Generalized Calabi--Eckmann threefolds}}
\label{cohomology_GCE}
It is well known that the intersection form of divisors of a Hirzebruch surface \(\hirz{q}\) is given by the matrix
\begin{equation}
    \label{hirz_q_intersection_matrix}
    M_q =
    \begin{bmatrix}
        0 & 1\\
        1 & -q
    \end{bmatrix}.
\end{equation}
Thus, \(\operatorname{tr}(M_q) = -q\) is a biholomorphic invariant of \(\hirz{q}\), whence these are all biholomorphically distinct. A fortiori, this obstructs the existence of any equivariant biholomorphism \(N_{q'} \to N_q\) if \(q' \neq q\). Nevertheless, by \cite[Prop.~18]{LoebNicolau1996}, any biholomorphism between these manifolds descends from a biholomorphism of \(\proj(U)\) sending the vector field \(X_{q'}\) to \(X_q\); hence
\begin{proposition}
    \(N_q \iso N_{q'}\) if and only if \(q = q' \in \mathbb{N}\).
\end{proposition}

Since the manifolds \(N_q\) are holomorphically distinct, we can still ask which cohomological theory---if any---perceives the differences in their complex structures; of course this is not de Rham cohomology, since the manifolds \(N_q\) are all diffeomorphic. As they are non-Kähler compact complex manifolds \cite{LoebNicolau1996,Meersseman2000}, there are basically two cohomology theories to consider: Dolbeault cohomology and Bott--Chern cohomology. These are respectively defined as
\begin{equation*}
    H^{s,t}_\delbar(N_q) = \frac{\ker(\delbar : \mathcal{A}^{s,t}_{N_q} \to \mathcal{A}^{s,t+1}_{N_q})}{\Ima (\delbar : \mathcal{A}^{s,t-1}_{N_q} \to \mathcal{A}^{s,t}_{N_q})}, \quad H^{s,t}_\text{BC}(N_q) = \frac{\ker(d : \mathcal{A}^{s,t}_{N_q} \to \mathcal{A}^{s+1,t}_{N_q} \oplus \mathcal{A}^{s,t+1}_{N_q})}{\Ima (\del\delbar : \mathcal{A}^{s-1,t-1}_{N_q} \to \mathcal{A}^{s,t}_{N_q})},
\end{equation*}
where \(\mathcal{A}^{s,t}_{N_q}\) is the vector space of \(\C\)-valued differential forms on \(N_q\) of type \((s,t)\).
\begin{remark}
    There also exist anti-Dolbeault and Aeppli cohomologies. The first is isomorphic to Dolbeault up to conjugation: namely \(H^{s,t}_\del(N_q) \iso \overline{H^{t,s}_\delbar(N_q)}\). Aeppli cohomology, instead, is defined in bidegree \((s,t)\) as
    \begin{equation*}
        H^{s,t}_\text{Ae}(N_q) = \frac{\ker(\del\delbar : \mathcal{A}^{s,t}_{N_q} \to \mathcal{A}^{s+1,t+1}_{N_q})}{\Ima (\del: \mathcal{A}^{s-1,t}_{N_q} \to \mathcal{A}^{s,t}_{N_q}) + \Ima (\delbar: \mathcal{A}^{s,t-1}_{N_q} \to \mathcal{A}^{s,t}_{N_q}) }.
    \end{equation*}
    By \cite{Schweitzer2007preprint}, in the compact case the antilinear Hodge-\(\star\)-operator induces an isomorphism
    \begin{equation*}
        H^{s,t}_\text{Ae}(N_q) \iso \overline{H^{3-s,3-t}_\text{BC}(N_q)}.
    \end{equation*}
    Therefore, anti-Dolbeault and Aeppli cohomologies do not carry more information about \(N_q\) than Dolbeault and Bott--Chern cohomologies.
\end{remark}

To compute these cohomologies for the manifolds \(N_q\) we use the model of \(\mathcal{A}^{s,t}_{N_q}\) provided by \cite[Prop.~2.8]{Thiella2026preprint}. We have a quasi-isomorphism of bigraded bidifferential algebras
\begin{equation}
    \label{the_model}
    (\mathcal{A}^{\bullet,\bullet}_{N_q},\del,\delbar) \simeq (\Lambda[\xi,\bar\xi] \otimes H^{*,*}_{\delbar}(\hirz{q}), \bar \delta, \delta) \nfd (\mathcal{M}^{\bullet,\bullet},\bar\delta, \delta)
\end{equation}
with \(\delta\xi = w_1+ i w_4\), \(\delta\bar\xi = 0\) and \(\delta|_{H^{*,*}_{\delbar}(\hirz{q})} =0\). Notice that, to be consistent with \cite{Thiella2026preprint}, under this quasi-isomorphism \(\delbar\) corresponds to \(\delta\).
\begin{figure}[p]
    \thisfloatpagestyle{empty}
    \begin{subfigure}{\linewidth}
        \makebox[\linewidth][c]{
            \begin{tikzpicture}[
                scale=1,
                >=stealth,
                dot/.style={circle, fill, inner sep=1.6pt},
                every label/.style={fill=white,inner sep=1pt},
                ]
                \useasboundingbox (0,0) rectangle (8,8);
                \draw[thick,step=2] (0,0) grid (8,8);
                
                \node[dot, label=below:$1$] at (1, 1) {};
                
                \node[dot, label=right:$\xi$] (xi) at (3.0, 1) {};
                \node[dot, label=left:$w_1 + i w_4$] (w1_minus) at (3.0, 2.5) {};
                \draw[thick] (xi) -- (w1_minus);
                
                \node[dot, label=left:$\bar{\xi}$] (xi_bar) at (1, 3.0) {};
                \node[dot, label=right:$w_1 - i w_4$] (w1_plus) at (2.5, 3.0) {};
                \draw[thick] (xi_bar) -- (w1_plus);
                
                \node[dot, label=right:$\alpha$] at (5, 2.5) {};
                
                \node[dot, label= left:{$\xi \wedge\bar{\xi}$}] (sq_bl) at (3.5, 3.5) {};
                \node[dot, label=below right:{$\xi\wedge(w_1 - i w_4)$}] (sq_br) at (4.5, 3.5) {};
                \node[dot, label=right:{$-q w_1\wedge w_4$}] (sq_tr) at (4.5, 4.5) {};
                \node[dot, label=left:{$\bar{\xi}\wedge(w_1 + i w_4)$}] (sq_tl) at (3.5, 4.5) {};
                \draw[thick] (sq_bl) -- (sq_br) -- (sq_tr) -- (sq_tl) -- (sq_bl);
                
                \node[dot, label=right:$\bar{\alpha}$] (alpha_bar_right) at (3, 5.5) {};
                
                \node[dot, label=left:$\beta$] (beta) at (5, 5.5) {};
                \node[dot, label=left:{$\nu$}] (beta_top) at (5, 7) {};
                \draw[thick] (beta) -- (beta_top);
                
                \node[dot, label=left:$\bar{\beta}$] (beta_bar) at (5.5, 5) {};
                \node[dot, label=right:{$\bar\nu$}] (beta_bar_right) at (6.5, 5) {};
                \draw[thick] (beta_bar) -- (beta_bar_right);
                
                \node[dot, label=right:{$\operatorname{vol}$}] at (7.0, 7) {};
                
                \node[label=right:$\bar \delta$] at (8,0) {};
                \node[label=above:$\delta$] at (0,8) {};
                
            \end{tikzpicture}
        }
        \caption{\(q > 0\)}
        \label{fig:the_model_q_not_0}
    \end{subfigure}
    \vfill
    \begin{subfigure}{\linewidth}
        \makebox[\linewidth][c]{
            \begin{tikzpicture}[
                scale=1,
                >=stealth,
                dot/.style={circle, fill, inner sep=1.6pt},
                every label/.style={fill=white,inner sep=1pt},
                ]
                \useasboundingbox (0,0) rectangle (8,8.5);
                \draw[thick,step=2] (0,0) grid (8,8);
                
                \node[dot, label=below:$1$] at (1, 1) {};
                
                \node[dot, label=right:$\xi$] (xi) at (3.5, 1) {};
                \node[dot, label=right:$w_1 + i w_4$] (w1_minus) at (3.5, 2.5) {};
                \draw[thick] (xi) -- (w1_minus);
                
                \node[dot, label=left:$\bar{\xi}$] (xi_bar) at (1, 2.5) {};
                \node[dot, label = below:$w_1 - i w_4$] (w1_plus) at (2.5, 2.5) {};
                \draw[thick] (xi_bar) -- (w1_plus);
                
                \node[dot, label=right:$\alpha$] (alpha) at (5, 3) {};
                
                \node[dot, label= above left:{$\xi \wedge\bar{\xi}$}] (xi_xibar) at (3, 3) {};
                \node[dot, label= right:{$\xi\wedge w_1$}] (xi_wedge) at (4.5, 3.5) {};
                \node[dot, label=right:{$w_1\wedge w_4$}] (w_1w_2) at (4.5, 4.5) {};
                
                \node[dot, label=right:$\bar{\alpha}$] (alpha_bar) at (3, 5.5) {};
                
                \draw[thick] (xi_xibar) -- (alpha_bar);
                \draw[thick] (xi_xibar) -- (alpha);
                \node[dot, label=left:{$\bar{\xi}\wedge w_1$}] (xibar_wedge) at (3.5, 4.5) {};
                \draw[thick] (xi_wedge) -- (w_1w_2);
                \draw[thick] (xibar_wedge) -- (w_1w_2);

                \node[dot, label=left:$\beta$] (beta) at (5, 5.5) {};
                \node[dot, label=left:{$\nu$}] (beta_top) at (5, 7) {};
                \draw[thick] (beta) -- (beta_top);
                
                \node[dot, label=left:$\bar{\beta}$] (beta_bar) at (5.5, 5) {};
                \node[dot, label=right:{$\bar\nu$}] (beta_bar_right) at (6.5, 5) {};
                \draw[thick] (beta_bar) -- (beta_bar_right);
                
                \node[dot, label=right:{$\operatorname{vol}$}] at (7.0, 7) {};
                
                \node[label=right:$\bar \delta$] at (8,0) {};
                \node[label=above:$\delta$] at (0,8) {};
                
            \end{tikzpicture}
        }
        \caption{\(q = 0\)}
        \label{fig:the_model_q_0}
    \end{subfigure}
    \caption{A representation of \((\Lambda[\xi,\bar\xi] \otimes H^{*,*}_{\delbar}(\hirz{q}), \bar\delta,\delta)\) at the varying of \(q\)}
    \label{fig:the_model}
\end{figure}
\Cref{fig:the_model} provides a graphical representation of the bigraded bidifferential structure of \((\mathcal{M}^{\bullet,\bullet},\bar\delta,\delta)\), with the generators given as:
\begin{alignat*}{3}
    \alpha &= \xi \wedge \left(w_1 - \frac{i}{1-iq} w_4\right), &\quad \beta &= \frac{1}{2} \xi \wedge \bar \xi \wedge \big((q-i) w_1 + w_4\big),\\
    \nu &= \bar\xi \wedge w_1 \wedge w_4, &\quad  \operatorname{vol} &= \frac{i}{2} \xi\wedge \bar \xi \wedge w_1 \wedge w_4.
\end{alignat*}

By construction, the cohomology of \((\mathcal{M}^{s,\bullet},\delta)\) computes the Dolbeault cohomology of \(N_q\); in our diagram, this corresponds to considering only the vertical direction: the non-trivial Dolbeault classes are indeed represented by the dots that are not endpoints of any vertical segment. We thus obtain the Dolbeault cohomology groups generated as follows
\begin{equation*}
    H^{\bullet,\bullet}_{\delbar}(N_q) =
    \begin{array}{ccccccc}
        &&&1\\
        && 0 & & \bar\xi\\
        & 0 && w_1 && 0\\
        0 & & \alpha & & \bar\alpha&&0\\
        &0 && \beta && 0\\
        && \bar\nu&& 0\\
        &&& \operatorname{vol}
    \end{array}
\end{equation*}
where the \((s+t)\)-th row is sorted as \(H^{s+t,0}_\delbar, H^{s+t-1,1}_\delbar, \dots, H^{0,s+t}_\delbar\). In particular \(H^{*,*}_{\delbar}(N_q)\) does not depend on \(q\), as proven in general in \cite[\S13.2]{Hofer1993}.
\subsubsection{De Rham cohomology}
In the diagrams of \cref{fig:the_model} we can also recognize the generators of the de Rham cohomology of \(N_q\). Directly from the model \((\mathcal{M}^{\bullet,\bullet},\bar\delta,\delta)\), we already know that nontrivial de Rham cohomology groups are
\begin{equation*}
    H^0_\text{dR}(N_q) \iso \R\left\langle 1 \right\rangle,\qquad H^3_\text{dR}(N_q) \iso \R\left\langle\theta_1 \wedge w_1, (q \theta_1 + \theta_2) \wedge w_4\right\rangle, \qquad H^6(N_q)_\text{dR} \iso \R\left\langle \operatorname{vol} \right\rangle,
\end{equation*}
with \(\theta_1 = \frac{1}{2}(\xi + \bar\xi)\), \(\theta_2 = -\frac{i}{2}(\xi - \bar\xi)\) and
\begin{equation*}
    \operatorname{vol} = \theta_1 \wedge \theta_2 \wedge w_1 \wedge w_4 = \frac{i}{2} \xi\wedge \bar \xi \wedge w_1 \wedge w_4.
\end{equation*}
Since \(\theta_1 \wedge w_4 - \theta_2 \wedge w_1\) vanishes in \(H^3_\text{dR}(N_q)\), the de Rham classes \([\alpha]_\text{dR}\) and \([\bar\alpha]_\text{dR}\) can be written as
\begin{equation}
    \label{expressions_of_alpha}
    \begin{aligned}
        [\alpha]_\text{dR} &= \theta_1 \wedge w_1 + \frac{2 + iq}{2(q^2 + 1)} (q \theta_1 + \theta_2) \wedge w_4\\
        [\bar\alpha]_\text{dR} &= \theta_1 \wedge w_1 + \frac{2 - iq}{2(q^2 + 1)} (q \theta_1 + \theta_2) \wedge w_4,
    \end{aligned}
\end{equation}
which generate \(H^3_\text{dR}(N_q)\otimes \C\) when \(q > 0\).
This is no longer true for \(q = 0\), since in this case \([\alpha]_\text{dR} = [\bar\alpha]_\text{dR}\); a second independent generator of \(H^3_\text{dR}(N_0)\otimes\C\) is then provided by \((\xi + \bar\xi) \wedge w_1\). Note that this is reflected in \cref{fig:the_model_q_0}, where \(\alpha\) and \(\bar\alpha\) become part of a \emph{zigzag}.
\begin{remark}
    \label{real_cohomology_model}
    Applying the Leray--Serre spectral sequence directly to \(\cpxTorus{1} \hookrightarrow N_q \to \hirz{q}\), we obtain a model \((\mathcal{M}_\Z^\bullet,d_{\mathcal{M}})\) for the singular cohomology of \(N_q\). Namely, it is given by
    \begin{equation*}
        (\mathcal{M}_\Z^\bullet, d_{\mathcal{M}}) \dfn (\Lambda[\theta_1, \theta_2] \otimes H^*(\hirz{q};\Z), d)
    \end{equation*}
    with \(\theta_1 = \frac{1}{2}(\xi + \bar\xi)\), \(\theta_2 = -\frac{i}{2} (\xi-\bar\xi)\) and \(d\theta_1 = w_1\), \(d\theta_2 = w_4\). Notice that tensoring this with \(\C\) and taking the projection onto bidegrees provides the bigraded model \(\mathcal{M}^{\bullet,\bullet}\), and \(\mathcal{M}_\Z^\bullet\) is then the restriction to integral classes.

    In particular, the classes \(\theta_1 \wedge w_1\) and \((q \theta_1 + \theta_2) \wedge w_4\) are integral; and for \(q > 0\) we can rewrite the generators of \(H^3(N_q;\Z)\) as
    \begin{equation*}
        H^3(N_q;\Z) \iso \left\langle \Re\left( \frac{q - 2i}{q} \alpha\right), \Im\left(2\frac{q^2 +1}{q} \alpha\right)\right\rangle.
    \end{equation*}
\end{remark}

With a similar technique, we can compute the Bott--Chern cohomology. For this we have to check the zigzags in \cref{fig:the_model}, of which the outer vertices provide nontrivial Bott--Chern classes. This yields

\begin{equation*}
    \begin{array}{r@{\quad}ccccccc}
        &&&& 1 &&& \\
        &&& 0 && 0 && \\
        && 0 && w_1, w_4 && 0 & \\
        H^{\bullet,\bullet}_{\text{BC}}(N_q) = & 0 && \alpha && \bar{\alpha} && 0 \\
        && 0 && w_1 \wedge w_4 && 0 & \\
        &&& \bar{\nu} && \nu && \\
        &&&& \text{vol} &&& \\[2ex] %
        &&&& q = 0 &&&
    \end{array}
    \quad %
    \begin{array}{r@{\quad}ccccccc}
        &&&& 1 &&& \\
        &&& 0 && 0 && \\
        && 0 && w_1, w_4 && 0 & \\
        H^{\bullet,\bullet}_{\text{BC}}(N_q) = & 0 && \alpha && \bar{\alpha} && 0 \\
        && 0 && 0 && 0 & \\
        &&& \bar{\nu} && \nu && \\
        &&&& \text{vol} &&& \\[2ex] 
        &&&& q > 0 &&&
    \end{array}
\end{equation*}
in particular \(h^{2,2}_\text{BC} = 1\) if \(q = 0\), and \(h^{2,2}_\text{BC} = 0\) for every \(q > 0\). This can already be  observed in \cref{fig:the_model}: indeed, while at the center of \cref{fig:the_model_q_0}  we have a zig-zag with \(w_1 \wedge w_4\) as outer vertex, in \cref{fig:the_model_q_not_0} its nontrivial multiple \(- q w_1 \wedge w_4\) becomes part of a square: consequently, it no longer contributes to cohomology. Roughly speaking, this suggests that, for \(q > 0\), \(w_1 \wedge w_4\) defines a torsion class of order \(q\), However, as we will see, this intuition can be misleading; making it precise requires a refined version of Bott--Chern cohomology with coefficients in a ring where \(q\) is not a unit.

\section{Integral Bott--Chern cohomology}
In \cite{Schweitzer2007preprint}, Bott--Chern cohomology is defined as the hypercohomology of a complex of sheaves. There, for any complex manifold \(X\), the integral Bott--Chern complex is defined as follows:
\begin{equation*}
    \mathcal{B}_{s,t;\Z}^\bullet :
    \begin{cases}
        \Z(s) \xrightarrow{Z} \OO \oplus \overline{\OO}  \to \cdots \to \Omega^{s-1} \oplus \overline{\Omega^{s-1}} \to \overline{\Omega^s} \to \cdots \to \overline{\Omega^{t-1}} \to 0 & \text{if } s \leq t\\
        \Z(s) \xrightarrow{Z} \OO \oplus \overline{\OO}  \to \cdots \to \Omega^{s-1} \oplus \overline{\Omega^{s-1}} \to {\Omega^s} \to \cdots \to {\Omega^{t-1}} \to 0 &\text{if } s > t
    \end{cases}
\end{equation*}
where \(\OO = \OO_X\) and \(\Omega^k = \Omega^k_X\) denote the sheaf of holomorphic functions of \(X\) and the sheaf of holomorphic \(k\)-forms, respectively. Moreover, \(\Z(s) = (2\pi i)^s \Z \hookrightarrow \C\) is the Tate twist of the sheaf of locally constant functions. The map \(Z = (\iota,-\iota)\) is given by the natural inclusion \(\iota : \Z(s) \hookrightarrow \OO\) in the first component, and by \(-\iota : \Z(s) \hookrightarrow \overline{\OO}\) in the second one. Therefore, the \emph{integral Bott--Chern} cohomology group of bidegree \((s,t)\) is
\begin{equation*}
    H^{s,t}_\text{BC}(X;\Z) = \mathbb{H}^{s+t}(X,\mathcal{B}^\bullet_{s,t;\Z}),
\end{equation*}
where \(\mathbb{H}^{s+t}\) denotes the \((s+t)\)-th hypercohomology group of a complex of sheaves.
\begin{remark}
    One can analogously define the Bott--Chern cohomology with coefficients in any subring \(R \hookrightarrow \C\)  by replacing \(\Z(s)\) with \(R\) and changing \(Z\) accordingly.  Of course, the inclusion of rings induces an injection of the Bott--Chern complexes \(\mathcal{B}^\bullet_{s,t;R} \hookrightarrow \mathcal{B}^\bullet_{s,t;\C}\) that, in turn, induces a natural map in cohomology \(H^{s,t}_\text{BC}(X;R) \to H^{s,t}_\text{BC}(X;\C)\) that is not injective in general.
\end{remark}
In particular, for \(R = \C\) we have the Bott--Chern complex with coefficients in \(\C\). By \cite[Prop.~4.2,4.3]{Schweitzer2007preprint}, this complex computes the usual Bott--Chern cohomology; more precisely
\begin{equation*}
    H^{s,t}_\text{BC}(X) \iso  \mathbb{H}^{s+t}(X;\mathcal{B}^\bullet_{s,t;\C}) = H^{s,t}_\text{BC}(X;\C).
\end{equation*}

\subsection{The relation with usual cohomology}
\label{compatibility_between_BC_cohomologies}
In certain cases, singular and usual Bott--Chern cohomologies can be combined to compute integral Bott--Chern cohomology groups. Indeed, we observe that the natural inclusion \(\mathcal{B}^\bullet_{s,t;\Z} \hookrightarrow \mathcal{B}^\bullet_{s,t;\C}\) of complexes of sheaves fits in the short exact sequence
\begin{lrbox}{\tikzcdBox}
    \begin{tikzcd}
        0 \ar[r] & \mathcal{B}^\bullet_{s,t;\Z} \ar[r] & \mathcal{B}^\bullet_{s,t;\C} \ar[r] & \frac{\C}{\Z} \ar[r] & 0,
    \end{tikzcd}
\end{lrbox}
\begin{equation}
    \label{exact_sequence_integral_complex-BC}
    \usebox{\tikzcdBox}
\end{equation}
where \(\C/\Z\) is regarded as the locally constant sheaf concentrated in degree \(0\). This yields an induced long exact sequence in hypercohomology
\begin{equation*}
    \begin{tikzcd}[column sep = small]
        \cdots \ar[r] & H^{k-1}\left(X;\frac{\C}{\Z}\right) \ar[r] &  \mathbb{H}^{k}(X;\mathcal{B}^\bullet_{s,t;\Z}) \ar[r] & \mathbb{H}^{k}(X;\mathcal{B}^\bullet_{s,t;\C}) \ar[r] & H^{k}\left(X;\frac{\C}{\Z}\right) \ar[r] & \cdots
    \end{tikzcd}
\end{equation*}
By the Universal Coefficients Theorem for singular cohomology,
\begin{equation*}
    H^{k}\left(X;\tfrac{\C}{\Z}\right) \iso H^{k}(X;\Z) \otimes_{\Z} \tfrac{\C}{\Z} \oplus \operatorname{Ext}^1\left(H_{k-1}(X),\tfrac{\C}{\Z}\right).
\end{equation*}
Since \(\C/\Z\) is a divisible group, the second summand is always trivial, and thus
\begin{equation*}
    H^{k}\left(X;\tfrac{\C}{\Z}\right) \iso \frac{H^{k}(X;\C)}{H^{k}(X;\Z)}.
\end{equation*}
In particular, we have
\begin{proposition}
    \label{0_and_top_BC_cohomology}
    Let \(X\) be an \(n\)-dimensional compact complex manifold. If \(X\) is connected and \(H^{2n-1}(X;\Z) = 0\), then \(H^{0,0}_\text{BC}(X;\Z) \iso H^{n,n}_\text{BC}(X;\Z) \iso \Z\).
\end{proposition} 
\begin{proof}
    Since \(\mathcal{B}^\bullet_{0,0;\Z} = \Z\), we trivially have \(H^{0,0}_\text{BC}(X;\Z) \iso \Z\). On the other hand, as by compactness \(H^{2n}(X;\Z) \iso \Z\), in total degree \(2n\) the hypercohomology long exact sequence described above gives
    \begin{equation*}
        \begin{tikzcd}
            0 \ar[r] & H^{n,n}_\text{BC}(X,\Z) \ar[r] & \C \ar[r] & \frac{\C}{\Z} \ar[r] & 0
        \end{tikzcd}
    \end{equation*}
    as \(H^{2n-1}\left(X;\frac{\C}{\Z}\right) \iso H^{2n-1}(X;\Z) \otimes_\Z \frac{\C}{\Z} = 0\). The statement then follows immediately.
\end{proof}
\begin{remark}
    In general, the top degree Bott--Chern cohomology is not isomorphic to \(\Z\). Indeed, if \(\kappa : H^{2n-1}(X;\Z) \otimes \frac{\C}{\Z} \to H^{n,n}_\text{BC}(X,\Z)\) is the connecting morphism in degree \(2n-1\), we only have the extension of groups
    \begin{equation*}
        \begin{tikzcd}
            0 \ar[r] & \Ima\kappa \ar[r] & H^{n,n}_\text{BC}(X,\Z) \ar[r] & \Z \ar[r] & 0.
        \end{tikzcd}
    \end{equation*}
    Since the first is a quotient of a divisible group, it is divisible as well; thus the extension is trivial and so
    \begin{equation*}
        H^{n,n}_\text{BC}(X,\Z) \iso \Z \oplus \Ima\kappa.
    \end{equation*}
\end{remark}
In particular, for a compact and connected complex manifold, only the \emph{reduced component} of the top integral Bott--Chern cohomology group is isomorphic to \(\Z\) in general.

\subsection{A soft resolution of the Bott--Chern complex}
\label{soft_resolution}
In many cases, the exact sequence above cannot determine the integral Bott--Chern cohomology groups of middle degree: for this we need a more direct approach. Namely, we need to consider an acyclic resolution of the Bott--Chern complex. Before discussing it, we recall some basic constructions in homological algebra. For this background our main reference is \cite{Weibel1994}.

First we introduce the following notation: for any chain complex \((A^\bullet,d_A)\) we denote by \(A^{\bullet \geq s}\) the dumb filtration
\begin{equation*}
    A^s \to A^{s+1} \to \cdots;
\end{equation*}
the quotient \(\sigma_s A^\bullet = A^\bullet/A^{\bullet\geq s}\) is a chain complex with no terms above degree \(s-1\). Therefore, we can compactly write the integral Bott--Chern complex as
\begin{equation*}
    \mathcal{B}_{s,t;\Z}^\bullet = \Z(s) \xrightarrow{Z} \sigma_s \Omega^\bullet \oplus \sigma_t \overline{\Omega^\bullet},
\end{equation*}
for \(s,t \geq 0\). We can interpret \(\Z(s)\) as a chain complex concentrated in degree \(0\), whereby \(Z\) defines a morphism of chain complexes
\begin{equation*}
    \begin{tikzcd}[column sep = small]
        \cdots \ar[r] & 0 \ar[r] \ar[d] & \cdots \ar[r]  & 0 \ar[r]  \ar[d] & \Z(s) \ar[r]  \ar[d, "Z"] & 0 \ar[r]  \ar[d] &  \cdots \ar[r] & 0 \ar[r]   \ar[d] & \cdots\\
        \cdots \ar[r] & 0\ar[r] & \cdots \ar[r]  & 0 \ar[r] & \OO \oplus \overline{\OO} \ar[r]& \Omega^1 \oplus \overline{\Omega^1} \ar[r] & \cdots \ar[r] & \Omega^k \oplus \overline{\Omega^k} \ar[r]& \cdots
    \end{tikzcd}
\end{equation*}

A second construction we need is the \emph{shift}: given a chain complex \((A^\bullet, d_A)\) we define its \emph{shift} \(A^\bullet[k]\) as
\begin{equation*}
    A^\bullet[k] = A^{\bullet - k} \quad \text{and} \quad d_{A[k]} = (-1)^k d_A.
\end{equation*}
According to this choice of signs, positive values of \(k\) yield a shift \emph{to the right}, while negative values induce a shift \emph{to the left}. It is not difficult to check that shifting chain complexes defines an endofunctor of the category of chain complexes; indeed this appears in the definition of \emph{mapping cone} \cite[\S 1.5]{Weibel1994}.
\begin{definition}
    Let \(f : (A^\bullet,d_A) \to (B^\bullet,d_B)\) be a morphism of chain complexes. The \emph{mapping cone} of \(f\) is the chain complex \((\operatorname{cone}(f)^\bullet, d_f)\) with
    \begin{equation*}
        \operatorname{cone}(f)^n = A^{n+1} \oplus B^n = (A^\bullet[-1] \oplus B^\bullet)^n,
    \end{equation*}
    and the differential given by the matrix
    \begin{equation*}
        \begin{bmatrix}
            -d_A & 0 \\
            -f & d_B
        \end{bmatrix}.
    \end{equation*}
\end{definition}
\begin{remark}
    We chose this sign convention to match that of bicomplexes. Indeed, as their squares \emph{anticommute}, this is the natural choice so that the mapping cone of a chain map \(f : A^\bullet \to B^\bullet\) defines a double complex
    \begin{equation*}
        \begin{tikzcd}
            \cdots \ar[r] & A^\ell \ar[r,"-d_A^\ell"] \ar[d,"-f^{\ell+1}"'] & A^{\ell+1} \ar[d,"-f^{\ell+1}"] \ar[r] & \cdots\\
            \cdots \ar[r] & B^\ell \ar[r, "d_B^\ell"] & B^{\ell+1} \ar[r] &\cdots
        \end{tikzcd}
    \end{equation*}
    Nevertheless, choosing anticommutative squares for double complexes breaks the trivial identification with chain complexes of chain complexes (whose squares are commutative instead). To restore it, we need to introduce alternating signs: namely, if \(((A^{\bullet_1}, d_{A^{\bullet_1}})^{\bullet_2}, f^{\bullet_2})\) is a chain complex of chain complexes, the associated bicomplex is \((A^{\bullet,\bullet}, d_h, d_v)\), with \(A^{k,\ell} = (A^{\ell})^k\) and the differentials 
    \begin{equation*}
        d_h^{k,\ell} = d_{A^k}^\ell : A^{k,\ell} \to A^{k,\ell+1}, \quad\text{while} \quad d_v^{k,\ell} = (-1)^\ell f^\ell : A^{k,\ell} \to A^{k+1,\ell}.
    \end{equation*}
    Keeping in mind this identification, henceforth we will not draw a sharp distinction between complexes (and so in particular morphisms) of chain complexes and bicomplexes.
\end{remark}
\begin{definition}
    A chain bicomplex \((C^{\bullet,\bullet}, d_h, d_v)\) determines two chain complexes:
    \begin{equation*}
        \operatorname{Tot}^\Pi(C)^k = \prod_{p+q = k} C^{p,q}, \qquad \operatorname{Tot}^\oplus(C)^k = \bigoplus_{p+q = k} C^{p,q}
    \end{equation*}
    endowed with the natural differentials. The identity \(d^2 = 0\) follows from \(d_h \circ d_v + d_v \circ d_h = 0\). Notice that, whenever the anti-diagonals of \(C^{\bullet,\bullet}\) are bounded, the two constructions of total complexes coincide. In particular this holds for bicomplexes contained in a quadrant, as well as for those contained in a strip.
\end{definition}

We can immediately observe that for any morphism of chain complexes \(f : A^\bullet \to B^\bullet\), \(\operatorname{cone}(f)^\bullet[1] = \operatorname{Tot}(A^\bullet \xrightarrow{f} B^\bullet)\); in particular \(\mathcal{B}^\bullet_{s,t;\Z} = \operatorname{cone}(-Z)^\bullet[1]\). This observation yields an acyclic resolution of \(\mathcal{B}^\bullet_{s,t;\Z}\) assembled from an acyclic resolution of \(\Z(s)\) and one for \(\sigma_s \Omega^\bullet \oplus \sigma_t \overline{\Omega^\bullet}\). This is done as follows: first we consider the complex \((\mathcal{I}^\bullet_{\Z}, d_{\Z})\) of \emph{locally integral currents} on \(X\); we do not need their precise definition (for it see \cite{WuXiaojun2023ItaC} and references therein). A sufficient intuition is given by considering them as generalized singular chains with integral coefficients; of course their sheaf admits the inclusion \(\Z \hookrightarrow \mathcal{I}^\bullet_\Z\). Since we need coefficients in \(\Z(s)\), we take the Tate twist
\begin{equation*}
    \Z(s) \hookrightarrow \mathcal{I}_{\Z(s)}^\bullet = \mathcal{I}_\Z^\bullet \otimes \Z(s).
\end{equation*}
By \cite[Lemma~1]{WuXiaojun2023ItaC}, \(\mathcal{I}_\Z^\bullet\) is a soft resolution of \(\Z\), and since \(\Z(s) \iso \Z\) as \(\Z\)-modules, \(\mathcal{I}_{\Z(s)}^\bullet\) is a soft resolution of \(\Z(s)\) as well. To construct a resolution of \(\sigma_s \Omega^\bullet \oplus \sigma_t\overline{\Omega^\bullet}\), we denote by \(\mathcal{D}'^{s,t}\) the group of complex currents on \(X\) of type \((s,t)\): for any \(s\), \((\mathcal{D}'^{s,\bullet},\delbar)\) is a fine resolution of \(\Omega^s\), as well as \((\mathcal{D}'^{\bullet,t},\del)\) is a fine resolution of \(\overline{\Omega^t}\) by complex conjugation. Therefore,
\begin{equation*}
    \sigma_s \Omega^{\bullet}  \hookrightarrow (\sigma_{s,\cdot} \mathcal{D}'^{\bullet,\bullet},\delbar), \quad\text{and} \quad \sigma_t \overline{\Omega^{\bullet}}  \hookrightarrow (\sigma_{\cdot,t} \mathcal{D}'^{\bullet,\bullet},\del)
\end{equation*}
are fine---and thus soft---resolutions of \(\sigma_s \Omega^{\bullet}\) and \(\sigma_t \overline{\Omega^{\bullet}}\) respectively. To obtain an acyclic resolution of \(\mathcal{B}_{s,t;\Z} = \Z(s) \hookrightarrow \sigma_s\Omega^\bullet \oplus \sigma_t \overline{\Omega^\bullet}\) we consider as in \cite[p. 3]{Teh2016} the chain complex \((\sigma_{s,t}\mathcal{D}'^\bullet, d_{s,t})\) defined as
\begin{equation*}
    \sigma_{s,t}\mathcal{D}'^k = \bigoplus_{\substack{i + j = k\\i < s}} \mathcal{D}'^{i,j} \oplus \bigoplus_{\substack{i + j = k\\j < t}} \mathcal{D}'^{i,j}
\end{equation*}
with \(d_{s,t} : \mathcal{D}'^{k} \to \mathcal{D}'^{k+1}\) given by the pair
\begin{equation*}
    d_{s,t}(x,y) = (\pi_{s,\cdot} \delbar x, \pi_{\cdot,t} \del y ),
\end{equation*}
where
\begin{equation*}
    \pi_{s,\cdot} : \mathcal{D}'^k \to \bigoplus_{\substack{i + j = k\\i < s}} \mathcal{D}'^{i,j}, \qquad \text{and} \quad \pi_{\cdot,t} : \mathcal{D}'^k \to \bigoplus_{\substack{i + j = k\\j < t}} \mathcal{D}'^{i,j}
\end{equation*}
are the natural projections. Thus, applying \cite[Prop.~A.3]{HarveyLawson2008}, as it is done in \cite[Prop.~2.9]{Teh2016}, to the diagram
\begin{equation*}
    \begin{tikzcd}
        \mathcal{I}_{\Z(s)}^\bullet \ar[r,"\Psi"] & \sigma_{s,t} \mathcal{D}'^{\bullet}\\
        \Z(s) \ar[u,hook] \ar[r,"Z"] & \sigma_s\Omega^\bullet \oplus \sigma_t \overline{\Omega^\bullet} \ar[u,hook]
    \end{tikzcd}
\end{equation*}
where \(\Psi\) is the natural projection on bigraded components, we obtain
\begin{align*}
    \mathbb{H}^k(X,\mathcal{B}^\bullet_{s,t;\Z}) &= \mathbb{H}^k(X,\Z(s) \hookrightarrow \sigma_s\Omega^\bullet \oplus \sigma_t \overline{\Omega^\bullet})\\
    &\iso \mathbb{H}^k(X,\operatorname{cone}(\Psi : \mathcal{I}_{\Z(s)}^\bullet \to \sigma_{s,t} \mathcal{D}'^{\bullet} ))\\
    &\iso H^k (\operatorname{cone}(\Gamma(X,\Psi) : \mathcal{I}_{\Z(s)}(X)^\bullet \to \sigma_{s,t} \mathcal{D}'^{\bullet}(X))).
\end{align*}

\section{Integral BC cohomology of generalized Calabi--Eckmann 3-folds}
In the case of generalized Calabi--Eckmann threefolds, the computation of integral Bott--Chern cohomology groups is trivial in most degrees. Indeed, since they are all homeomorphic to \(S^3 \times S^3\), their fifth Betti number vanishes; thus, by \Cref{0_and_top_BC_cohomology}, \(H^{0,0}_\text{BC}(X;\Z) \iso H^{3,3}_\text{BC}(X;\Z) \iso \Z\). Moreover, as \(\mathcal{B}^\bullet_{s,t;\Z} \iso \overline{\mathcal{B}^\bullet_{t,s;\Z}}\), integral Bott--Chern cohomology groups are symmetric in the exchange of \(s\) and \(t\); hence we can restrict to computing \(H^{s,t}_\text{BC}(N_q;\Z)\) for \(s \geq t\). In higher degrees, the exact sequence \labelcref{exact_sequence_integral_complex-BC} yields
\begin{equation*}
    \begin{tikzcd}
        H^{s+t-1}(N_q;\Z) \otimes_\Z \frac{\C}{\Z} \ar[r] &  H^{s,t}_\text{BC}(N_q,\Z) \ar[r] & H^{s,t}_\text{BC}(N_q,\C) \ar[r] & H^{s+t}(N_q;\Z) \otimes_\Z \frac{\C}{\Z}.
    \end{tikzcd}
\end{equation*}
Since \(H^{s+t-1}(N_q;\Z) = H^{s+t}(N_q;\Z) = 0\) for \(s+t \in \{2,5\}\), we immediately obtain  \(H^{s,t}_\text{BC}(N_q;\Z) \iso H^{s,t}_\text{BC}(N_q;\C)\) in these total degrees. The same also occurs in total degree \(1\). Indeed, as the map
\begin{equation*}
    H^{0,0}_\text{BC}(N_q;\C) \longrightarrow H^0\left(N_q;\tfrac{\C}{\Z}\right) \iso \frac{H^0\left(N_q;\C\right)}{H^0\left(N_q;\Z\right)}
\end{equation*}
is induced by the identity, it is surjective. Hence, in degree \(1\), the exact sequence \labelcref{exact_sequence_integral_complex-BC} splits as
\begin{equation*}
    \begin{tikzcd}
        0 \ar[r] &  H^{s,t}_\text{BC}(N_q,\Z) \ar[r] & H^{s,t}_\text{BC}(N_q,\C) \ar[r] & H^{1}(N_q;\Z) \otimes_\Z \frac{\C}{\Z} = 0,
    \end{tikzcd}
\end{equation*}
whence \(H^{s,t}_\text{BC}(N_q;\Z) \iso H^{s,t}_\text{BC}(N_q;\C) = 0\) also for \(s+t = 1\).

Using the same exact sequence argument, we can compute the integral Bott--Chern cohomology groups of total degree \(3\). In bidegree \((3,0)\) we have
\begin{equation*}
    \begin{tikzcd}
        0 = H^2\left(N_q;\frac{\C}{\Z}\right) \ar[r] & H^{3,0}_\text{BC}(N_q;\Z) \ar[r]& H^{3,0}_\text{BC}(N_q;\C) = 0,
    \end{tikzcd}
\end{equation*}
whence \(H^{3,0}_\text{BC}(N_q;\Z) = 0\). On the other hand, in bidegree \((2,1)\), we get
\begin{equation*}
    \begin{tikzcd}
        0 = H^2\left(N_q;\frac{\C}{\Z}\right) \ar[r] & H^{2,1}_\text{BC}(N_q;\Z) \ar[r]& H^{2,1}_\text{BC}(N_q;\C) \ar[r,"p \circ \epsilon"] & H^3\left(N_q;\frac{\C}{\Z}\right) \iso \frac{H^3(N_q;\C)}{H^3(N_q;\Z)};
    \end{tikzcd}
\end{equation*}
where \(\epsilon : H^{2,1}_\text{BC}(N_q;\C) \to H^3(N_q;\C)\) is the natural map to de Rham cohomology and \(p\) is the projection onto the quotient. As \(H^3(N_q;\Z)\) can be naturally identified with a subgroup of \(H^3(N_q;\C) \iso H^3(N_q;\Z) \otimes_\Z \C\), the exactness implies
\begin{equation*}
    H^{2,1}_\text{BC}(N_q;\Z) \iso \ker(p \circ \epsilon) = \epsilon^\inv (H^3(N_q;\Z)).
\end{equation*}
The kernel \(\epsilon^\inv (H^3(N_q;\Z))\) contains precisely the classes \(t \alpha\) such that \(\epsilon(t\alpha) \in H^3(N_q;\Z)\). By \cref{expressions_of_alpha}, 
\begin{equation*}
    \epsilon(t\alpha) = t \theta_1 \wedge w_1 + t \frac{2+iq}{2(q^2 +1)} (q \theta_1 + \theta_2) \wedge w_4;
\end{equation*}
this lies in \( H^3(N_q;\Z)\) only if the conditions
\begin{equation*}
    \begin{cases}
        t \in \Z\\
        t \frac{2+iq}{2(q^2 +1)} \in \Z
    \end{cases}
\end{equation*}
are both satisfied. As this admits a nontrivial solution only if \(q = 0\), we deduce that
\begin{equation*}
    H^{2,1}_\text{BC}(N_q;\Z) \iso \ker(p \circ \epsilon) =
    \begin{cases}
        \left\langle \alpha \right\rangle \iso \Z & \text{if } q = 0 \\
        0 & \text{if } q \neq 0.
    \end{cases}
\end{equation*}
\begin{remark}
    This behavior has no analogue in de~Rham cohomology. Indeed, by the de~Rham Theorem and the Universal Coefficients Theorem, the de Rham cohomology of a manifold is isomorphic to the singular cohomology with coefficients in \(\R\) (or \(\C\) in the complex case); thus any class can be rescaled so that it lies in the singular cohomology with integral coefficients. Here, instead, we have the Bott--Chern class \([\alpha]_\text{BC}\) that cannot be rescaled to any integral class when \(q > 0\). This happens because the canonical map
    \begin{equation*}
        H^{2,1}_\text{BC}(N_q;\C) \to H^3(N_q;\C),
    \end{equation*}
    that is injective by \cref{fig:the_model_q_not_0}, identifies the Bott--Chern cohomology group with a complex subspace of \(H^3(N_q;\C)\) that has trivial intersection with the image of the inclusion \(H^3(N_q;\Z) \hookrightarrow H^3(N_q;\C)\) induced by the Universal Coefficients Theorem. With a slight modification of this argument, the same behavior can be observed in Deligne cohomology as well. For this cohomological theory, see e.g. \cite{Brylinski2008}.
\end{remark}

\subsection{Computation of \texorpdfstring{\(H^{2,2}_\mathrm{BC}(N_q;\Z)\)}{H^22_BC(N_q;ZZ)}}
To compute the integral Bott--Chern cohomology in total degree \(4\) we employ the soft resolution of \(\mathcal{B}^\bullet_{s,t;\Z}\) we described in \Cref{soft_resolution}. This involves the chain complexes of global currents \(\mathcal{I}^\bullet_{\Z(s)}(N_q)\) and \(\mathcal{D}'^{s,t}(N_q)\) that are rarely finitely generated: this makes impractical to compute their cohomology directly. Nevertheless, by \cite[Thm.~5.11]{FedererFleming1960}, the inclusion of smooth singular cochains \( C_\infty^\bullet(N_q;\Z(s)) \hookrightarrow \mathcal{I}^\bullet_{\Z(s)}(N_q)\) induced by Poincaré duality is a quasi-isomorphism. Similarly, the inclusion \(\mathcal{A}^{\bullet,\bullet}_{N_q} \hookrightarrow \mathcal{D}'^{\bullet,\bullet}(N_q)\) of the bicomplex of smooth differential forms into the complex of currents is a quasi-isomorphism as well. Therefore, denoting by \(\tilde\Psi : C_\infty^\bullet(N_q;\Z(s)) \to \sigma_{s,t}\mathcal{A}^{\bullet}_{N_q}\) the counterpart of the map \(\Psi\) for forms, we have a square
\begin{equation*}
    \begin{tikzcd}
        C_\infty^\bullet(N_q;\Z(s)) \ar[r,"\tilde\Psi"] \ar[d,hook] & \sigma_{s,t}\mathcal{A}^{\bullet}_{N_q} \ar[d, hook]\\
         \mathcal{I}^\bullet_{\Z(s)}(N_q) \ar[r,"\Psi"] & \sigma_{s,t}\mathcal{D}'^{\bullet}(N_q)
    \end{tikzcd}
\end{equation*}
whose vertical arrows are quasi-isomorphisms of chain complexes of \(\Z\)-modules; moreover, it commutes because Poincaré duality provides a natural identification of differential forms as currents.

This reduces our problem to the computation of the cohomology of complexes of differential forms: for these we have the cohomological model described in \Cref{cohomology_GCE}. Indeed, we can replace \((\mathcal{A}^{\bullet,\bullet}_{N_q},\del,\delbar)\) with the model \((\mathcal{M}^{\bullet,\bullet}, \bar\delta, \delta)\) defined in \labelcref{the_model}. Similarly, a cohomological model for \(C_\infty^\bullet(N_q;\Z)\) is given by \((\mathcal{M}_\Z^\bullet, d_{\mathcal{M}})\) as described in \Cref{real_cohomology_model}. A priori, this is just a model for \(C^\bullet_\infty(N_q;\R)\), however, a straightforward computation of the Leray--Serre spectral sequence of the bundle \(N_q \to \hirz{q}\) shows that its singular cohomology is free. Thus, restricting it to forms that are dual to the singular chains of \(N_q\) provides a model for \(C^\bullet_\infty(N_q;\Z)\). Taking the Tate twist, we obtain a model \(\mathcal{M}_{\Z(s)}^\bullet\) for \(C^\bullet_\infty(N_q;\Z(s))\).

By construction, \(\tilde\Psi\) lifts to a map \(\mathcal{M}_\Z^\bullet \to \sigma_{s,t}\mathcal{M}^{\bullet}\) which, for simplicity, we will keep denoting by \(\tilde\Psi\). Therefore, the quasi-isomorphism \(\tilde\Psi \simeq \Psi\) induces a quasi-isomorphism of the corresponding cone complexes, namely:
\begin{equation*}
    \mathcal{B}_{s,t;\Z}^\bullet \simeq \operatorname{cone}(-Z)^\bullet[1] \simeq \operatorname{cone}(-\Psi)^\bullet[1] \simeq \operatorname{cone}(-\tilde\Psi)^\bullet[1].
\end{equation*}
In terms of \(\mathcal{M}^{\bullet,\bullet}\), the chain complex \(\operatorname{cone}(-\tilde\Psi)^\bullet[1]\) is the total complex of the following bicomplex
\begin{lrbox}{\tikzcdBox}
    \begin{tikzcd}
        \mathcal{M}^0_{\Z(s)} \ar[d, "\tilde\Psi^0"] \ar[r,"-d_{\mathcal{M}}"] & \mathcal{M}^1_{\Z(s)} \ar[d, "\tilde\Psi^1"] \ar[r,"-d_{\mathcal{M}}"]& \mathcal{M}^2_{\Z(s)} \ar[d, "\tilde\Psi^2"] \ar[r,"-d_{\mathcal{M}}"] & \mathcal{M}^3_{\Z(s)} \ar[d, "\tilde\Psi^3"] \ar[r,"-d_{\mathcal{M}}"] & \cdots\\
        \sigma_{s,t}\mathcal{M}^0 \ar[r, "\delta_{s,t}^0"] & \sigma_{s,t}\mathcal{M}^1 \ar[r,"\delta_{s,t}^1"] & \sigma_{s,t}\mathcal{M}^2 \ar[r, "\delta_{s,t}^2"] & \sigma_{s,t}\mathcal{M}^3 \ar[r,"\delta_{s,t}^3"] & \cdots
    \end{tikzcd}    
\end{lrbox}
\begin{equation}
    \label{computingBCmodel}
    \usebox{\tikzcdBox}
\end{equation}
Employing this, we can directly compute the remaining integral Bott--Chern cohomology groups \(H^{s,t}_\text{BC}(N_q;\Z)\) with \(s+t = 4\). We start from the most involved case of \(H^{2,2}_\text{BC}(N_q;\Z)\). By construction, its computation can be reduced to \(H^4(\operatorname{Tot}(\operatorname{cone}(-\tilde\Psi))^\bullet[1])\). In the degrees of our interest, this complex can be succinctly written as
\begin{equation*}
    A^\bullet = A^3 \xrightarrow{d_A^3} A^4 \xrightarrow{d_A^4} A^5
\end{equation*}
where, as \(\mathcal{M}^{k,0} = \mathcal{M}^{0,k} = \{0\}\) for \(k > 1\),
\begin{align*}
    A^3 &= \mathcal{M}_{\Z(2)}^3 \oplus \mathcal{M}^{1,1} \oplus \mathcal{M}^{1,1}\\
    A^4 &= \mathcal{M}_{\Z(2)}^4 \oplus \mathcal{M}^{1,2} \oplus \mathcal{M}^{2,1}\\
    A^5 &= \mathcal{M}_{\Z(2)}^5 \oplus \mathcal{M}^{1,3} \oplus \mathcal{M}^{3,1}
\end{align*}
while, as \(\pi^{1,3} = 0 = \pi^{3,1}\), the differentials are given as the matrices
\begin{equation*}
    d_A^3 =
    \begin{bmatrix}
        -d_{\Z} & 0 & 0\\
        -\pi^{1,2} & \delta & 0\\
        \pi^{2,1} & 0 & \bar\delta
    \end{bmatrix},
    \quad
    d_A^4 =
    \begin{bmatrix}
        -d_{\Z} & 0 & 0\\
        0 & \delta & 0\\
        0 & 0 & \bar\delta
    \end{bmatrix}.
\end{equation*}
We can verify that, as \(\Z\)-modules
\begin{align*}
    A^3 &= \left\langle \theta_1 \wedge w_1, \theta_1 \wedge w_4, \theta_2 \wedge w_1, \theta_2 \wedge w_4 \right\rangle \otimes_\Z \Z(2) \oplus \left\langle w_1,w_4,\xi\wedge\bar\xi \right\rangle \otimes_\Z \C\\
    A^4 &= \left\langle w_1 \wedge w_4, w_1 \wedge \theta_1 \wedge \theta_2, w_4 \wedge \theta_1 \wedge \theta_2 \right\rangle \otimes_\Z \Z(2) \oplus \left\langle \bar\xi \wedge w_1, \bar\xi \wedge w_4 \right\rangle \otimes_\Z \C\\
    &\quad \oplus \left\langle \xi \wedge w_1, \xi \wedge w_4 \right\rangle \otimes_\Z \C.
\end{align*}
A direct computation of the differentials \(d_A^3\) and \(d_A^4\) yields:
\begin{align*}
    \ker d_A^4 &= \left\langle w_1 \wedge w_4\right\rangle \otimes_\Z \Z(2) \oplus \left\langle \bar\xi \wedge w_1, \bar\xi \wedge w_4 \right\rangle \otimes_\Z \C \oplus \left\langle \xi \wedge w_1, \xi \wedge w_4 \right\rangle \otimes_\Z \C,\\
    \Ima d_A^3 &= \Big\langle \frac{1}{2} \xi \wedge w_1 - \frac{1}{2}\bar\xi \wedge w_1, - w_1 \wedge w_4 -\frac{i}{2}\xi \wedge w_1 - \frac{i}{2}\bar\xi \wedge w_1,\\
    &\quad -w_1 \wedge w_4 + \frac{1}{2}\xi \wedge w_4 - \frac{1}{2}\bar\xi \wedge w_4, q w_1 \wedge w_4 -\frac{i}{2}\xi \wedge w_4 - \frac{i}{2}\bar\xi \wedge w_4 \Big\rangle \otimes_\Z \Z(2)\\
    &\quad \oplus \left\langle \bar\xi \wedge (w_1 + i w_4) \right\rangle \otimes_\Z \C \oplus \left\langle \xi \wedge (w_1 - i w_4)\right\rangle \otimes_\Z \C
\end{align*}
Thus, after standard manipulations of the generators, we obtain
\begin{equation*}
    \frac{\ker d_A^4}{\Ima d_A^3} \iso \frac{\left\langle w_1 \wedge w_4\right\rangle \otimes_\Z \Z(2) \oplus \left\langle \bar\xi \wedge w_1, \xi \wedge w_1 \right\rangle \otimes_\Z \C}{\left\langle q w_1 \wedge w_4 , \frac{1}{2} \xi \wedge w_1 - \frac{1}{2}\bar\xi \wedge w_1, w_1\wedge w_4  +\frac{i}{2}\xi \wedge w_1 + \frac{i}{2}\bar\xi \wedge w_1 \right\rangle \otimes_\Z \Z(2)}.
\end{equation*}
Setting \(\digamma \dfn w_1 \wedge w_4\), \(\lambda \dfn \frac{1}{2} \xi \wedge w_1 - \frac{1}{2}\bar\xi \wedge w_1\), and \(\mu \dfn \frac{i}{2}\xi \wedge w_1 + \frac{i}{2}\bar\xi \wedge w_1\), we have
\begin{equation}
    \label{H22explicit}
    H^{2,2}_\text{BC}(N_q;\Z) \iso \frac{\ker d_A^4}{\Ima d_A^3} \iso \frac{\left\langle\digamma\right\rangle\otimes_\Z \Z(2) \oplus \left\langle\lambda,\mu\right\rangle \otimes_\Z \C}{\left\langle q \digamma, \lambda, \digamma + \mu\right\rangle \otimes_\Z \Z(2)} \iso \frac{\C}{\Z(2)} \oplus \frac{\C}{q \Z(2)}.
\end{equation}

\subsection{Computation of \texorpdfstring{\(H^{3,1}_\mathrm{BC}(N_q;\Z)\)}{H^31_BC(N_q;ZZ)}}

As above, in order to compute \(H^{3,1}_\text{BC}(N_q;\Z)\) we consider the following chain complex:
\begin{equation*}
    C^\bullet = C^3 \xrightarrow{d_C^3} C^4 \xrightarrow{d_C^4} C^5
\end{equation*}
where,
\begin{align*}
    C^3 &= \mathcal{M}_{\Z}^3 \oplus \mathcal{M}^{1,1}\\
    C^4 &= \mathcal{M}_{\Z}^4 \oplus \mathcal{M}^{1,2} \oplus \mathcal{M}^{2,1}\\
    C^5 &= \mathcal{M}_{\Z}^5 \oplus \mathcal{M}^{2,2} ,
\end{align*}
and the differentials given by the matrices
\begin{equation*}
    d_C^3 =
    \begin{bmatrix}
        -d_{\Z} & 0 \\
        \pi^{1,2} & \delta\\
        \pi^{2,1}  & \bar\delta
    \end{bmatrix},
    \quad
    d_C^4 =
    \begin{bmatrix}
        -d_{\Z} & 0 & 0\\
        \pi^{2,2} & \bar\delta & \delta\\
    \end{bmatrix}.
\end{equation*}
As \(\Z\)-modules 
\begin{align*}
    C^3 &= \left\langle \theta_1 \wedge w_1, \theta_1 \wedge w_4, \theta_2 \wedge w_1, \theta_2 \wedge w_4 \right\rangle \oplus \left\langle w_1,w_4,\xi\wedge\bar\xi \right\rangle \otimes_\Z \C,\\
    C^4 &= \left\langle w_1 \wedge w_4, w_1 \wedge \theta_1 \wedge \theta_2, w_4 \wedge \theta_1 \wedge \theta_2 \right\rangle \oplus \left\langle \bar\xi \wedge w_1, \bar\xi \wedge w_4 \right\rangle \otimes_\Z \C \oplus \left\langle \xi \wedge w_1, \xi \wedge w_4 \right\rangle \otimes_\Z \C,
\end{align*}
and
\begin{align*}
    \ker d_C^4 &= \Big\langle w_1 \wedge w_4 + \frac{i}{2}(\xi - \bar\xi) \wedge w_1, w_1 \wedge w_4 - \frac{1}{2}(\xi+\bar\xi) \wedge w_4, q w_1 \wedge w_4 - \frac{i}{2}(\xi - \bar\xi) \wedge w_4  \Big\rangle\\
    &\qquad \oplus \left\langle (\bar\xi - \xi) \wedge w_1 \right\rangle \otimes_\Z \C,\\
    \Ima d_A^3 &= \Big\langle \frac{1}{2} (\xi +\bar\xi) \wedge w_1, - w_1 \wedge w_4 + \frac{1}{2}(\xi + \bar \xi) \wedge w_4, - w_1 \wedge w_4 - \frac{i}{2}( \xi- \bar\xi) \wedge w_1,\\
    &\quad q w_1 \wedge w_4 -\frac{i}{2}(\xi - \bar\xi) \wedge w_4 \Big\rangle \oplus \left\langle \bar\xi \wedge (w_1 + i w_4) \right\rangle \otimes_\Z \C \oplus \left\langle \xi \wedge (w_1 - i w_4)\right\rangle \otimes_\Z \C.
\end{align*}
Simplifying the quotient we find
\begin{equation*}
    H^{3,1}_\text{BC}(N_q;\Z) \iso \frac{\ker d_C^4}{\Ima d_C^3} \iso \{0\}.
\end{equation*}
Therefore, we can represent the integral Bott--Chern cohomology of \(N_q\) as the Hodge diamond:
\begin{equation}
    \label{Hodge_diamond_integral_BC_N_q}
    \begin{array}{r@{\quad}ccccccc}
        &&&& \Z &&& \\
        &&& 0 && 0 && \\
        && 0 && \C^2 && 0 & \\
        H^{\bullet,\bullet}_{\text{BC}}(N_0;\Z) = & 0 && \Z && \Z && 0 \\
        && 0 && \frac{\C}{\Z} \oplus \C && 0 & \\
        &&& \C && \C && \\
        &&&& \Z &&& \\[2ex] %
        &&&& q = 0 &&&
    \end{array}
    \quad %
    \begin{array}{r@{\quad}ccccccc}
        &&&& \Z &&& \\
        &&& 0 && 0 && \\
        && 0 && \C^2 && 0 & \\
        H^{\bullet,\bullet}_{\text{BC}}(N_q;\Z) = & 0 && 0 && 0 && 0 \\
        && 0 && \frac{\C}{\Z} \oplus \frac{\C}{q\Z} && 0 & \\
        &&& \C && \C && \\
        &&&& \Z &&& \\[2ex] 
        &&&& q > 0 &&&
    \end{array}
\end{equation}
Since for any \(q > 0\) \(\frac{\C}{q\Z} \iso \C^*\) as complex abelian Lie groups, we have
\begin{theorem}
    \label{integral_BC_are_iso}
    If \(q,q'\) are positive integers, \(H^{s,t}_\text{BC}(N_q;\Z) \iso H^{s,t}_\text{BC}(N_{q'};\Z)\) for any \(s,t\).
\end{theorem}
Therefore, the integral Bott--Chern cohomology only distinguishes the standard Calabi--Eckmann threefold from all the other generalized ones.

\subsection{The ring structure of \(H^{*,*}_\mathrm{BC}(N_q;\Z)\)}
When a cohomology theory admits an algebra structure, its multiplicative structure typically provides a more refined invariant than the cohomology groups considered individually. In fact, the integral Bott--Chern cohomology can be endowed with such a structure by considering the natural multiplication morphism
\begin{equation*}
    \mathcal{B}^\bullet_{s,t;\Z} \otimes_{\Z} \mathcal{B}^\bullet_{s',t';\Z} \to \mathcal{B}^\bullet_{s+s',t+t';\Z}
\end{equation*}
induced by the wedge product in each degree; with a suitable choice of signs, the differentials fulfil the Leibniz rule \cite{WuXiaojun2023ItaC}. For any complex manifold \(X\), the natural multiplication of the corresponding \v{C}ech hypercocycles induces a graded ring structure on \(H^{*,*}_\text{BC}(X;\Z)\). An explicit expression of the multiplication law can be found in \cite{Schweitzer2007preprint,WuXiaojun2023ItaC}.

Moreover, the choice of signs made in the construction is such that the isomorphism as bigraded \(\Z\)-modules \(H^{s,t}_\text{BC}(X; \Z) \iso H^{s,t}_\text{BC}(X)\) identifies the abstract cup product with the wedge product of \((s,t)\)-forms. Hence, by naturality, there is a commutative diagram
\begin{lrbox}{\tikzcdBox}
    \begin{tikzcd}
        H^{s,t}_\text{BC}(X;\Z) \otimes H^{s',t'}_\text{BC}(X;\Z)  \ar[r,"\smile"] \ar[d,hook] & H^{s+s',t+t'}_\text{BC}(X;\Z) \ar[d,hook]\\
        H^{s,t}_\text{BC}(X;\C) \otimes H^{s',t'}_\text{BC}(X) \ar[r,"\wedge"] & H^{s+s',t+t'}_\text{BC}(X)
    \end{tikzcd}
\end{lrbox}
\begin{equation}
    \label{cup_product_compatibility}
    \usebox{\tikzcdBox}
\end{equation}
mapping the integral Bott--Chern cohomology ring into the usual Bott--Chern cohomology. Although generalized Calabi--Eckmann threefolds are all biholomorphically distinct, they cannot be distinguished by the integral Bott--Chern cohomology ring. Indeed, we have the following:
\begin{proposition}
    Let \(N_q, N_{q'}\) be two generalized Calabi--Eckmann threefolds with \(q,q' > 0\). Then \(H^{*,*}_\text{BC}(N_q;\Z) \iso H^{*,*}_\text{BC}(N_{q'};\Z)\) as bigraded rings.
\end{proposition}
\begin{proof}
    From \labelcref{Hodge_diamond_integral_BC_N_q} and the compatibility between integral and usual Bott--Chern cohomologies, an isomorphism of bigraded \(\Z\)-modules \(f : H^{*,*}_\text{BC}(N_{q'};\Z) \to H^{*,*}_\text{BC}(N_{q};\Z)\) defines a ring isomorphism if it preserves the restriction of the cup product map to:
    \begin{equation*}
         H^{1,1}_\text{BC}(N_q;\Z) \otimes H^{1,1}_\text{BC}(N_q;\Z) \to H^{2,2}_\text{BC}(N_q;\Z).
    \end{equation*}
    Identifying \(\Z(s)\) with \(\Z\) in each degree, \(H^{*,*}_\text{BC}(N_q;\Z)\) can be represented as the diamond
    \begin{equation*}
        \begin{array}{ccccccc}
            &&& 1 &&& \\
            && 0 && 0 && \\
            & 0 && \left\langle w_1,w_4\right\rangle_\C && 0 & \\
            0 && 0 && 0 && 0 \\
            & 0 && \left\langle x \right\rangle_\frac{\C}{\Z} \oplus \left\langle y \right\rangle_\frac{\C}{q\Z}  && 0 & \\
            && \left\langle \bar\nu \right\rangle_\C && \left\langle \nu \right\rangle_\C && \\
            &&& \left\langle \operatorname{vol} \right\rangle_\Z &&&
        \end{array}
    \end{equation*}
    where we have adopted the more compact notation \(\left\langle\cdot\right\rangle_M \dfn \left\langle\cdot\right\rangle \otimes_\Z M\). For \(H^{*,*}_\text{BC}(N_{q'};\Z)\) we take a similar set of generators \(\{1', w_1',w_4', x',y', \bar\nu',\operatorname{vol}'\}\) together with their complex conjugates. We define an isomorphism of \(\Z\)-modules \(f : H^{*,*}_\text{BC}(N_{q};\Z) \to H^{*,*}_\text{BC}(N_{q'};\Z)\), on generators of bidegree \((s,t)\) with \(s \geq t\) as
    \begin{gather*}
        f(1) = 1',\quad f(w_1) = p w_1', \quad f(w_4) = p w_4',\\
        f(x) = x', \quad f(y) = p^2 y',\quad f(\bar\nu) = \bar\nu',\quad f(\operatorname{vol}) = \operatorname{vol}',
    \end{gather*}
    where \(p\) is such that \(p^2 = \frac{q'}{q}\). The non-trivial algebraic relations among the generators of both sets are:
    \begin{gather*}
        w_1\smile w_1 = 0,\quad  w_1 \smile w_4 = y, \quad w_4 \smile w_4 = -q y = 0 \in \left\langle y\right\rangle \otimes_\Z \frac{\C}{q \Z}.
    \end{gather*}
    The conclusion then follows observing that these are preserved by \(f\), as we now verify:
    \begin{align*}
        f(w_1 \smile w_4) &= f(y) = p^2 y' = p^2 w_1' \smile w_4' = f(w_1) \smile f(w_4),\\
        f(w_1 \smile w_1) &= 0 = p^2 w_1' \smile w_1' = f(w_1) \smile f(w_1),\\
        f(w_4 \smile w_4) &= 0 = p^2 w_4' \smile w_4' = f(w_4) \smile f(w_4).\qedhere
    \end{align*}
\end{proof}
By analogous choices, we can also construct an isomorphism of \(\C\)-algebras 
\begin{equation*}
    f_{\C} : H^{*,*}_\text{BC}(N_q;\C) \to H^{*,*}_\text{BC}(N_{q'};\C).
\end{equation*}
Indeed, to describe the \(\C\)-algebra structure of \(H^{*,*}_\text{BC}(N_q;\C)\), it is sufficient to also take in account the class \(\alpha \in H^{2,1}_\text{BC}(N_q;\C)\). This fulfils \(\alpha \wedge w_1 = \frac{\bar\nu}{q + i}\), \(\alpha \wedge w_4 = \frac{\bar\nu}{1-iq}\) and \(\alpha \wedge \bar\alpha = \frac{q}{1+q^2} \operatorname{vol}\). Therefore, extending \(f\) naturally to Bott--Chern cohomology with complex coefficients and setting
\begin{equation*}
    f_{\C}(\alpha) = \frac{1-iq'}{p(1-iq)} \alpha',\qquad f_{\C}(\bar\alpha) = \frac{1+iq'}{p(1+iq)} \bar\alpha'
\end{equation*}
we can check that
\begin{align*}
    f_\C(\alpha \wedge w_1) &= \frac{1}{q+i} f(\bar\nu) = \frac{1}{q+i} \bar\nu' = \frac{1-iq'}{1-iq} \alpha' \wedge w_1' = f_\C(\alpha) \wedge f_{\C}(w_1)\\
    f_\C(\alpha \wedge w_4) &= \frac{1}{1-iq} f(\bar\nu) = \frac{1}{1-iq} \bar\nu' = \frac{1-iq'}{1-iq} \alpha' \wedge w_1' = f_\C(\alpha) \wedge f_{\C}(w_4)\\
    f_\C(\alpha \wedge \bar\alpha) &= \frac{q}{1+q^2} f(\operatorname{vol}) = \frac{q}{1+q^2} \operatorname{vol}' = \frac{q (1+q'^2)}{q'(1+q^2)} \alpha' \wedge \bar\alpha' = f(\alpha) \wedge f(\bar\alpha).
\end{align*}
Therefore:
\begin{theorem}
    \label{BC_ring_iso_and_algebraa}
    If both \(q,q' > 0\), there exist a ring isomorphism \(f : H^{*,*}_\text{BC}(N_{q'};\Z) \to H^{*,*}_\text{BC}(N_{q};\Z)\) and a \(\C\)-algebra isomorphism \(f_{\C} : H^{*,*}_\text{BC}(N_q;\C) \to H^{*,*}_\text{BC}(N_{q'};\C)\) such that the square
    \begin{equation*}
        \begin{tikzcd}
            H^{*,*}_\text{BC}(N_{q'};\Z) \ar[r,"f","\iso"'] \ar[d] & H^{*,*}_\text{BC}(N_{q};\Z) \ar[d]\\
            H^{*,*}_\text{BC}(N_{q'};\C) \ar[r,"f_\C","\iso"'] & H^{*,*}_\text{BC}(N_{q};\C)
        \end{tikzcd}
    \end{equation*}
    whose vertical arrows are the natural maps induced by \(\mathcal{B}^\bullet_{s,t;\Z} \hookrightarrow \mathcal{B}^\bullet_{s,t;\C}\), commutes.
\end{theorem}
This result reveals a peculiar behavior of the cohomology of generalized Calabi--Eckmann threefolds. Indeed, while we know that there is no biholomorphism \(N_q \to N_{q'}\) for \(q \neq q'\), none of the cohomological theories considered here, when taken individually, detects the non-existence of such a biholomorphism. Even the integral Bott--Chern cohomology, which seemed the best candidate to detect such a difference in the complex structure fails to do so. The reason is that, apart from those of total degree \(0\) or \(6\), all the integral Bott--Chern cohomology groups are divisible;  as divisible abelian groups admit no canonical discrete subgroup, they are rich in automorphisms that can be exploited to construct isomorphisms like the one above. As we shall see, a finer invariant can be extracted by considering these cohomological theories not individually, but in relation to one another.

\begin{remark}
    We may ask whether a refined version of the Universal Coefficients Theorem can be exploited to recover information about the map
    \begin{equation*}
        H^{s,t}_\text{BC}(X;\Z) \longrightarrow H^{s,t}_\text{BC}(X;\C)
    \end{equation*}
    connecting the integral Bott--Chern cohomology of a complex manifold \(X\) with the usual Bott--Chern cohomology. The answer is surprisingly nontrivial. Indeed, on a compact complex manifold the sheaf (hyper)cohomology functor commutes with all filtered colimits \cite[Thm.~4.12.1]{Godement1998}, thus we can apply the result in \cite{Kahn2024preprint} to obtain such a generalization of the Universal Coefficients Theorem. Nevertheless, this does not behave as one might naively expect. Indeed, for the Bott--Chern complex, we obtain the short exact sequence
    \begin{equation*}
        \begin{tikzcd}
            0 \ar[r] & \mathbb{H}^k(X, \mathcal{B}^\bullet_{s,t;\Z}) \otimes_\Z \C \ar[r] &  \mathbb{H}^k(X, \mathcal{B}^\bullet_{s,t;\Z} \otimes_\Z \C ) \ar[r] & \mathbb{H}^{k+2}(X,\operatorname{Tor}(\mathcal{B}^\bullet_{s,t;\Z},\C)) = 0
        \end{tikzcd}
    \end{equation*}
    where the last term vanishes as \(\C\) is flat over \(\Z\). However, \(\mathcal{B}^\bullet_{s,t;\Z} \otimes_\Z \C\) is very far from being quasi-isomorphic to \(\mathcal{B}^\bullet_{s,t;\C}\). This can be observed by computing, for example, the first cohomology sheaf:
    \begin{equation*}
        \mathcal{H}^1(\mathcal{B}^\bullet_{s,t;\C}) = \frac{\ker (\OO \oplus \overline{\OO} \xrightarrow{\del \oplus \delbar} \Omega^1 \oplus \overline{\Omega^1})}{\Ima (\underline\C \xrightarrow{(1,-1)} \OO \oplus \overline{\OO})} \iso \frac{\underline{\C}^2}{\underline\C(e_1 - e_2)} \iso \underline\C;
    \end{equation*}
    whose stalks are thus \(1\)-dimensional complex vector spaces. On the other hand,
    \begin{equation*}
        \mathcal{H}^1(\mathcal{B}^\bullet_{s,t;\Z}) = \frac{\ker (\OO \oplus \overline{\OO} \xrightarrow{\del \oplus \delbar} \Omega^1 \oplus \overline{\Omega^1})}{\Ima (\underline{\Z(s)} \xrightarrow{(1,-1)}\OO \oplus \overline{\OO})} \iso \frac{\underline\C^2}{\underline\Z(e_1-e_2)}
    \end{equation*}
    as \(\Z\)-modules. Since \((-) \otimes_\Z \C\) is exact, one has
    \begin{equation*}
        \mathcal{H}^1(\mathcal{B}^\bullet_{s,t;\Z} \otimes_\Z \C) \iso \mathcal{H}^1(\mathcal{B}^\bullet_{s,t;\Z}) \otimes_\Z \C \iso \frac{\underline\C^2 \otimes_\Z \C}{\underline\Z(e_1-e_2) \otimes_\Z \C}.
    \end{equation*}
    Since \(\C^2 \otimes_{\Z} \C\) is a complex vector space whose dimension is the cardinality of the continuum, while \(\Z(e_1-e_2) \otimes_\Z \C\) is a complex line, the stalks of \(\mathcal{H}^1(\mathcal{B}^\bullet_{s,t;\Z} \otimes_\Z \C)\) cannot be isomorphic to those of \(\mathcal{H}^1(\mathcal{B}^\bullet_{s,t;\C})\); this obstructs any quasi-isomorphism between \(\mathcal{B}^\bullet_{s,t;\Z} \otimes_\Z \C\) and \(\mathcal{B}^\bullet_{s,t;\C}\).    
\end{remark}
This is the kind of problem that the so-called \emph{Condensed Mathematics} (see, e.g. \cite{Scholze2026preprintCondensedMath}) aims to solve. In this approach to mathematics, the categories of topological objects (such as topological spaces, topological groups, etc.) are replaced by their condensed counterparts, which are better behaved from a categorical perspective. For example, while the category of topological abelian groups is not abelian, that of \emph{condensed} abelian groups is. Even though the theory of condensed sheaves over analytic spaces is still a matter of research, it seems reasonable that, in the condensed setting, a change of scalars like the one described above behaves as expected:
namely, that \(\mathcal{B}^\bullet_{s,t;\C}\) can be recovered from \(\mathcal{B}^\bullet_{s,t;\Z}\) via extension of scalars, through an adjoint equivalence involving the condensed tensor product. This is impossible in the standard algebraic setting, since, as we have seen, the restriction of scalars completely destroys the `continuous structure' of complex vector spaces, as evidenced by the failure of \(\mathcal{B}^\bullet_{s,t;\Z} \otimes_\Z \C \simeq \mathcal{B}^\bullet_{s,t;\C}\) shown above.

\subsection{A cohomological invariant for Generalized Calabi--Eckmann threefolds}
The fact that none of the cohomological theories we have investigated distinguishes generalized Calabi--Eckmann threefolds when considered alone does not imply that a combination thereof must fail as well. Indeed, by its very definition, integral Bott--Chern cohomology provides a natural bridge between the other cohomological theories we introduced above. In fact, the integral Bott--Chern complex fits into the following exact sequence of chain complexes of sheaves
\begin{equation*}
    \begin{tikzcd}
        0 \ar[r] & (\sigma_s\Omega^\bullet \oplus \sigma_t\overline{\Omega^\bullet})[1] \ar[r] & \mathcal{B}^\bullet_{s,t;\Z} \ar[r] & \Z(s) \ar[r] & 0
    \end{tikzcd}
\end{equation*}
which can then be used to enrich the structure of integral Bott--Chern cohomology groups. In fact, in the specific case of generalized Calabi--Eckmann threefolds, we can consider the following piece of the corresponding long exact sequence in hypercohomology:
\begin{equation*}
    \begin{tikzcd}[column sep = tiny]
        \mathbb{H}^3(N_q,\mathcal{B}^\bullet_{2,2;\Z}) \ar[r] & H^3(N_q;\Z(2)) \ar[r, "c"] & \mathbb{H}^3(N_q,\sigma_2\Omega^\bullet \oplus \sigma_2\overline{\Omega^\bullet}) \ar[r] & H^{2,2}_\text{BC}(N_q;\Z) \ar[r] & H^4(N_q;\Z(2)).
    \end{tikzcd}
\end{equation*}
We can use \labelcref{computingBCmodel} to compute \(\mathbb{H}^3(N_q,\mathcal{B}^\bullet_{2,2;\Z})\): this time the problem reduces to computing the cohomology of the complex
\begin{equation*}
    A^\bullet = A^2 \xrightarrow{d_A^2} A^3 \xrightarrow{d_A^3} A^4
\end{equation*}
where,
\begin{align*}
    A^2 &= \mathcal{M}_{\Z(2)}^2 \oplus (\mathcal{M}^{0,1} \oplus \mathcal{M}^{1,0}) \oplus (\mathcal{M}^{1,0} \oplus \mathcal{M}^{0,1})\\
    A^3 &= \mathcal{M}_{\Z(2)}^3 \oplus \mathcal{M}^{1,1} \oplus \mathcal{M}^{1,1}\\
    A^4 &= \mathcal{M}_{\Z(2)}^4 \oplus \mathcal{M}^{1,2} \oplus \mathcal{M}^{2,1}
\end{align*}
and the differentials are given as the matrices
\begin{equation*}
    d_A^2 =
    \begin{bmatrix}
        -d_{\Z} & 0 & 0\\
        -\pi^{1,1} & (\bar\delta,\delta) & 0\\
        \pi^{1,1} & 0 & (\bar\delta,\delta)
    \end{bmatrix},
    \quad
    d_A^3 =
    \begin{bmatrix}
        -d_{\Z} & 0 & 0\\
        -\pi^{1,2} & \delta & 0\\
        \pi^{2,1} & 0 & \bar\delta
    \end{bmatrix}
\end{equation*}
as \(\delta(\mathcal{M}^{0,1}) = 0\) and \(\bar\delta(\mathcal{M}^{1,0}) = 0\). Therefore,
\begin{align*}
    A^2 &= \left\langle w_1, w_4, \theta_1 \wedge \theta_2\right\rangle \otimes_\Z \otimes \Z(2) \oplus \left\langle \xi, \bar\xi \right\rangle \otimes_\Z \C  \oplus\left\langle \xi, \bar\xi \right\rangle \otimes_\Z \C, \\
    A^3 &= \left\langle \theta_1 \wedge w_1, \theta_1 \wedge w_4, \theta_2 \wedge w_1, \theta_2 \wedge w_4 \right\rangle \otimes_\Z \Z(2) \oplus \left\langle w_1,w_4,\xi\wedge\bar\xi \right\rangle \otimes_\Z \C \oplus \left\langle w_1,w_4,\xi\wedge\bar\xi \right\rangle \otimes_\Z \C.
\end{align*}
To keep track of the repeated summands in \(A^3\) we introduce the symbols \(\epsilon, \varepsilon_1, \varepsilon_2\): the first spans \(\Z(2)\) as a \(\Z\)-module, while the others span \(\C\) as a complex vector space. Thus,
\begin{align*}
    \Ima d^2_A &= \Big\langle -w_1 \varepsilon_1 + w_1 \varepsilon_2, -w_4 \varepsilon_1 + w_4 \varepsilon_2,\\
    &\qquad (\theta_1 \wedge w_4 - \theta_2 \wedge w_1)\epsilon - \left(\frac{i}{2} \xi \wedge \bar\xi\right) \varepsilon_1 + \left(\frac{i}{2} \xi \wedge \bar\xi\right) \varepsilon_2 \Big\rangle \otimes_\Z \Z(2)\\
    &\quad \oplus \left\langle w_1 \varepsilon_1, w_4 \varepsilon_1, w_1 \varepsilon_2, w_4 \varepsilon_2 \right\rangle \otimes_\Z \otimes \C
\end{align*}
The cocycles, on the other hand, are given by
\begin{align*}
    \ker d_A^3 &= \left\langle (\theta_1 \wedge w_4 - \theta_2 \wedge w_1)\epsilon - \left(\frac{i}{2} \xi \wedge \bar\xi\right) \varepsilon_1 + \left(\frac{i}{2} \xi \wedge \bar\xi\right) \varepsilon_2 \right\rangle \otimes_\Z \Z(2) \oplus \left\langle \varpi_q \right\rangle \otimes_\Z \Z(2)\\
    &\qquad \oplus \left\langle w_1 \varepsilon_1, w_4\varepsilon_1, w_1 \varepsilon_2, w_4\varepsilon_2 \right\rangle \otimes_\Z \C
\end{align*}
with
\begin{equation*}
    \varpi_q =
    \begin{cases}
        (\theta_1 \wedge w_1 + \theta_2 \wedge w_4)\epsilon + \left(\frac{1}{2} \xi \wedge \bar\xi\right) \varepsilon_1 + \left(\frac{1}{2} \xi \wedge \bar\xi\right) \varepsilon_2 & \text{if } q = 0\\
        0 & \text{if } q > 0
    \end{cases}
\end{equation*}
Standard computations then yield
\begin{equation*}
    \mathbb{H}^3(N_q,\mathcal{B}^\bullet_{2,2;\Z}) \iso \frac{\ker d_A^3}{\Ima d_A^2} \iso 
    \begin{cases}
        \left\langle\varpi_0\right\rangle \otimes_{\Z} \Z(2) & \text{if } q = 0\\
        0 & \text{if } q > 0;
    \end{cases}
\end{equation*}
depending on the \(d\)-closedness of \(\varpi_q\). 

In particular, \(\ker c = 0\) for \(q > 0\); therefore, since \(H^4(N_q;\Z) = \{0\}\) and 
\begin{equation*}
    \mathbb{H}^3(N_q,\sigma_2\Omega^\bullet \oplus \sigma_2\overline{\Omega^\bullet}) \iso \mathbb{H}^3(N_q;\sigma_2\Omega^\bullet) \oplus \mathbb{H}^3(N_q;\sigma_2\overline{\Omega^\bullet}) \iso H^{1,2}_{\delbar}(N_q) \oplus H^{2,1}_{\del}(N_q) \iso \C^2,
\end{equation*}
the sequence above collapses to
\begin{equation*}
    \begin{tikzcd}
        0 \ar[r] & H^3(N_q;\Z) \ar[r, "c"] & H^{1,2}_{\delbar}(N_q) \oplus H^{2,1}_{\del}(N_q) \ar[r] & H^{2,2}_\text{BC}(N_q;\Z) \ar[r] & 0.
    \end{tikzcd}
\end{equation*}
On the other hand, for \(q = 0\), the projection \(\mathcal{B}^\bullet_{2,2;\Z} \to \Z(2)\) induces the map in cohomology
\begin{equation*}
    \mathbb{H}^3(N_q,\mathcal{B}^\bullet_{2,2;\Z}) \to H^3(N_q;\Z(2)) \iso H^3(N_q;\Z) \otimes_{\Z} \Z(2)
\end{equation*} 
that projects any hypercocycle onto its \(\epsilon\)-part. Thus, in this case, \(\ker c = \left\langle \theta_1 \wedge w_1 + \theta_2 \wedge w_4 \right\rangle \otimes_\Z \Z(2)\). In both cases, one can verify that the resulting group is consistent with the computation in \Cref{H22explicit}. In fact, for \(q > 0\), from the exact sequence we have
\begin{equation*}
    H^{2,2}_\text{BC}(N_q;\Z) \iso \frac{H^{1,2}_{\delbar}(N_q) \oplus H^{2,1}_{\del}(N_q)}{H^3(N_q;\Z)} \iso \frac{\C \oplus \C}{\Z^2},
\end{equation*}
that is isomorphic to \(\frac{\C^2}{\Z(2) \oplus q\Z(2)}\). On the other hand, for \(q = 0\), \(\Ima c\) has rank \(1\), whence
\begin{equation*}
    H^{2,2}_\text{BC}(N_0;\Z) \iso \frac{H^{1,2}_{\delbar}(N_0) \oplus H^{2,1}_{\del}(N_0)}{\Ima c} \iso \C \oplus \frac{\C}{\Z},
\end{equation*}
as was already established. However, the exact sequence above enriches the structure of the integral Bott--Chern cohomology in such a way that \(H^{2,2}_\text{BC}(N_q;\Z)\) can be assigned a characteristic torsion class. This is described in the proof of the following result.
\begin{theorem}
    \label{GCE_are_distinct_cohomological}
    Let \(N_q\) and \(N_{q'}\) be two generalized Calabi--Eckmann threefolds; let \(c,c'\) be the maps in cohomology constructed as above. Then \(c \iso c'\) if and only if \(q = q'\).
\end{theorem}
\begin{proof}
    We consider the short exact sequence
    \begin{equation*}
        \begin{tikzcd}
            0 \ar[r] & \operatorname{Coim}c \ar[r] & H^{1,2}_{\delbar}(N_q) \oplus H^{2,1}_{\del}(N_q) \ar[r] & H^{2,2}_\text{BC}(N_q;\Z) \ar[r] & 0,
        \end{tikzcd}
    \end{equation*}
    where 
    \begin{equation*}
        \operatorname{Coim}c = \frac{H^3(N_q;\Z(2))}{\ker c} \iso
        \begin{cases}
            \left\langle \theta_1 \wedge w_1 \right\rangle & \text{if } q = 0,\\
            \left\langle \theta_1 \wedge w_1, (q\theta_1 + \theta_2) \wedge w_4 \right\rangle & \text{if } q > 0.
        \end{cases}
    \end{equation*}
    Since \(c = (\pi^{1,2}, - \pi^{2,1})\) and \(\theta_1 \wedge w_4 - \theta_2 \wedge w_1\) is \(d\)-exact, the coimage is mapped through \(c\) to the group
    \begin{equation*}
        c(\operatorname{Coim} c) = \left\langle \frac{1}{2} \bar\xi \wedge w_1 - \frac{1}{2} \xi \wedge w_1, \frac{1}{2}(iq -1) \bar\xi \wedge w_1 + \frac{1}{2} (iq + 1) \xi \wedge w_1\right\rangle \subseteq H^{1,2}_{\delbar}(N_q) \oplus H^{2,1}_{\del}(N_q).
    \end{equation*}
    We fix \(\{\frac{1}{2} \bar\xi \wedge w_1, \frac{1}{2} \xi \wedge w_1\}\) as a basis for \(H^{1,2}_{\delbar}(N_q) \oplus H^{2,1}_{\del}(N_q) \iso \C^2\); with respect to it, we can see \(H^{2,2}_\text{BC}(N_q,\Z)\) as the quotient \(\C^2/\Gamma_q\) where \(\Gamma_q\) is the \(\Z\)-module generated by \( \{e_1 - e_2, (iq-1) e_1 + (iq + 1) e_2\}\). The isomorphism of maps \(c \iso c'\), yields an isomorphism of their images \(\Gamma_q\) and \(\Gamma_{q'}\) as subgroups of \(\C^2\). This may happen only if there exists \(A \in \GL(2,\Z[i])\) such that
    \begin{equation*}
        M_{\Gamma_{q'}} \dfn
        \begin{bmatrix}
            1 & iq' - 1\\
            -1 & iq' + 1
        \end{bmatrix}
        =
        \begin{bmatrix}
            1 & iq - 1\\
            -1 & iq + 1
        \end{bmatrix}
        A \nfd M_{\Gamma_q} A.
    \end{equation*}
    Since \(|\det A| = 1\), 
    \begin{equation*}
        4 q'^2 = |\det M_{\Gamma_{q'}}|^2 = |\det M_{\Gamma_{q}}|^2 = 4 q^2,
    \end{equation*}
    whence \(q = q'\) since they are both taken non-negative. The other implication is trivial.
\end{proof}

\printbibliography

\noindent\begin{minipage}{\linewidth}
    \begin{center}
        \small
        \rule{4cm}{.5pt}
        \bigskip
        
        Università degli Studi di Firenze, Dipartimento di Matematica e Informatica ``Ulisse Dini'',\\ V.le Morgagni 67/a, 50134 Firenze, Italia
        
        email: \href{mailto:federico.thiella@unifi.it}{federico.thiella@unifi.it}
    \end{center}
\end{minipage}

\end{document}

%% file: preamble.tex
\usepackage[utf8]{inputenc}
\usepackage[T1]{fontenc}
\usepackage[english]{babel}
\usepackage[margin=1in]{geometry}

\usepackage{amsmath}
\usepackage{amsthm}
\usepackage{amsfonts}
\usepackage{amssymb}
\usepackage{mathtools}
\usepackage{tikz}
\usepackage{nicematrix}
\usepackage{leftindex}
\usepackage{faktor}
\usepackage{subcaption}
\usepackage{floatpag}

\usepackage{csquotes}
\usepackage{accents}
\usepackage[babel]{microtype}
\usepackage{enumerate}
\usepackage{lipsum}

\usepackage{graphicx}
\usepackage{xcolor}
\usepackage{hyperref}
\usepackage[textsize = scriptsize, color=green]{todonotes}
\usetikzlibrary{cd,arrows.meta,math,positioning}
\usepackage{standalone}

\usepackage[backend=biber,style=numeric,  eprint=false, sorting=nyt,giveninits=true]{biblatex}
\AtEveryBibitem{%
\ifentrytype{article}{%
\clearfield{url}%
\clearfield{urldate}%
}{}
}
\AtEveryBibitem{%
\ifentrytype{book}{%
\clearfield{url}%
\clearfield{urldate}%
}{}
}

\newcommand{\R}{\mathbb{R}}
\newcommand{\C}{\mathbb{C}}
\newcommand{\CP}{\mathbb{C}\mathrm{P}}
\newcommand{\Z}{\mathbb{Z}}

\newcommand{\OO}{\mathcal{O}}

\newcommand{\aff}{\mathbb{A}}

\newcommand{\del}{\partial}
\newcommand{\delbar}{{\bar\partial}}
\newcommand{\iso}{\cong}
\newcommand{\inv}{{-1}}

\newcommand{\proj}{\mathbb{P}}
\newcommand{\dfn}{\coloneqq}
\newcommand{\nfd}{\eqqcolon}

\let\Re\relax
\DeclareMathOperator{\Re}{\mathrm{Re}}
\let\Im\relax
\DeclareMathOperator{\Im}{\mathrm{Im}}
\newcommand{\hirz}[1]{\mathbb{F}_{#1}}
\newcommand{\calT}{\mathcal{T}}

\newcommand{\cpxTorus}[1]{\mathbb{T}^{#1}}

\DeclareMathOperator{\Ima}{Im}

\newcommand{\trans}[1]{\leftindex^{t}{#1}}

\DeclareMathOperator{\GL}{GL}

\let\svthefootnote\thefootnote
\newcommand\freefootnote[1]{%
  \let\thefootnote\relax%
  \footnotetext{#1}%
  \let\thefootnote\svthefootnote%
}

\newtheorem{theorem}{Theorem}
\newtheorem{proposition}{Proposition}[section]

\theoremstyle{definition}
\newtheorem{definition}{Definition}
\newtheorem{remark}[proposition]{Remark}
\theoremstyle{remark}

\newtheoremstyle{named}{}{}{\itshape}{}{\bfseries}{.}{.5em}{#1 \thmnote{#3}}
\theoremstyle{named}
\newtheorem*{namedtheorem}{Theorem}

\newtheoremstyle{named_example}{}{}{}{}{\bfseries}{.}{.5em}{#1Example \thmnumber{#2}\,: \thmnote{#3}}
\theoremstyle{named_example}

\usepackage{zref-clever}
\newcommand{\cref}[1]{\zcref{#1}}
\newcommand{\Cref}[1]{\zcref[S]{#1}}

\newcommand{\labelcref}[1]{\zcref[noname]{#1}}

\zcsetup{nameinlink=false,abbrev}

\AddToHook{env/lemma/begin}{%
  \zcsetup{countertype={proposition=lemma}}}
\AddToHook{env/remark/begin}{%
  \zcsetup{countertype={proposition=remark}}}
\AddToHook{env/named_example/begin}{%
  \zcsetup{countertype={proposition=named_example}}}

\zcRefTypeSetup{named_example}{
Name-sg = Example ,
name-sg = example ,
Name-pl = Examples ,
name-pl = examples ,
}

\newsavebox{\tikzcdBox}